\documentclass[11pt, a4paper, reqno]{amsart}
\DeclareMathAlphabet{\mathscrbf}{OMS}{mdugm}{b}{n}

\usepackage[utf8]{inputenc}
\usepackage[slovene, english]{babel}
\usepackage{polski}
\usepackage{amsmath,amsfonts,amssymb,mathrsfs,amsthm,amscd, ulem}
\usepackage[pdftex]{graphicx}
\usepackage{tikz-cd}
\usepackage{soul} 
\usepackage{MnSymbol}
\usepackage[all,cmtip]{xy}
\usepackage{enumitem}

\usepackage{slashed}

\usepackage{pgfplots}

\usepackage[mathcal]{euscript}  

\usepackage{url}

\usetikzlibrary{arrows}

 \usepackage{relsize}

\tikzset{
  no line/.style={draw=none,
    commutative diagrams/every label/.append style={/tikz/auto=false}},
  from/.style args={#1 to #2}{to path={(#1)--(#2)\tikztonodes}}}

\title[Moduli of oriented formal groups and complex bordism]{
Moduli stack of oriented formal groups\\ and periodic complex bordism
}
\author[Rok Gregoric]{Rok Gregoric}
\thanks{University of Texas at Austin}
\address{Department of Mathematics, University of Texas at Austin, Austin, TX 78712, USA}
\email{gregoric@math.utexas.edu}
\date{\today}
\newtheorem{theorem}{Theorem}[subsection]
\newtheorem{theoremm}{Theorem}
\newtheorem{motto}[theoremm]{Motto}

\newtheorem{corollary}[theorem]{Corollary}
\newtheorem{corollarry}[theoremm]{Corollary}
\newtheorem{lemma}[theorem]{Lemma}
\newtheorem{prop}[theorem]{Proposition}
\newtheorem{claim}[theorem]{Claim}
\theoremstyle{definition}
\newtheorem{definition}[theorem]{Definition}
\newtheorem{definnition}[theoremm]{Definition}
\newtheorem{variant}[theorem]{Variant}
\newtheorem{cons}[theorem]{Construction}
\newtheorem{conns}[theoremm]{Construction}

\newtheorem{ex}[theorem]{Examples}
\newtheorem{exammple}[theoremm]{Example}

\newtheorem{remark}[theorem]{Remark}
\newtheorem{remmark}[theoremm]{Remark}

\usepackage{tikz, calc}
\usetikzlibrary{matrix,arrows}

\newcommand*{\Set}{{\mathcal S\text{et}}}

\newcommand*{\Cat}{\mathcal C\mathrm{at}_\infty}

\newcommand*{\CAlg}{{\operatorname{CAlg}}}

\newcommand*{\mC}{\mathcal C}
\newcommand*{\mD}{\mathcal D}
\newcommand*{\mO}{\mathcal O}

\newcommand*{\mI}{\mathcal I}

\newcommand*{\mX}{\mathcal X}

\newcommand*{\mS}{\mathcal S}
\newcommand*{\sL}{\mathscr L}

\newcommand*{\sO}{\mathcal O}
\newcommand*{\sF}{\mathscr F}

\newcommand*{\E}{\mathbb E_\infty}

\newcommand*{\heart}{\heartsuit}

\newcommand*{\sheafhom}{\mathscr{H}\kern -.5pt om}
\DeclareMathOperator{\Novak}{\mathscr{N}\text{\kern -3pt {\calligra\large ovak}}\,\,}

\usepackage{amsmath,calligra,mathrsfs}
\DeclareMathOperator{\fHom}{\mathscr{H}\text{\kern -3pt {\calligra\large om}}\,}

\DeclareMathOperator{\Hom}{\operatorname{Hom}}

\DeclareMathOperator{\bDelta}{\boldsymbol\Delta}

\DeclareMathOperator{\Sp}{\operatorname{Sp}}

\DeclareMathOperator{\Fun}{\operatorname{Fun}}

\DeclareMathOperator{\Spec}{\operatorname{Spec}}
\DeclareMathOperator{\Spf}{\operatorname{Spf}}

\DeclareMathOperator{\Shv}{\mathcal S\mathrm{hv}}

\DeclareMathOperator{\colim}{\operatorname{colim}}

\DeclareMathOperator{\Map}{\operatorname{Map}}
\DeclareMathOperator{\QCoh}{\operatorname{QCoh}}
\DeclareMathOperator{\IndCoh}{\operatorname{IndCoh}}

\DeclareMathOperator{\Vect}{\operatorname{Vect}}

\DeclareMathOperator{\Tot}{\operatorname{Tot}}

\DeclareMathOperator{\Sym}{\operatorname{Sym}}

\DeclareMathOperator{\MP}{\mathrm{MP}}
\DeclareMathOperator{\MU}{\mathrm{MU}}
\DeclareMathOperator{\M}{\mathcal M^\mathrm{or}_\mathrm{FG}}
\DeclareMathOperator{\Mo}{\mathcal M^\heart_\mathrm{FG}}

\DeclareMathOperator{\G}{\mathbf G}
\DeclareMathOperator{\Z}{\mathbf Z}

\DeclareMathOperator{\CMon}{\operatorname{CMon}}

\DeclareMathOperator{\Mod}{\operatorname{Mod}}

\DeclareMathOperator{\cMod}{\operatorname{cMod}}

\renewcommand{\i}{\infty}

\renewcommand{\Pr}{\mathcal P\mathrm r}
\newcommand{\PrL}{\mathcal P\mathrm r^{\mathrm L}}

\newcommand{\w}{\widehat}

\renewcommand{\i}{\infty}

\DeclareFontFamily{U}{matha}{\hyphenchar\font45}
\DeclareFontShape{U}{matha}{m}{n}{
      <5> <6> <7> <8> <9> <10> gen * matha
      <10.95> matha10 <12> <14.4> <17.28> <20.74> <24.88> matha12
      }{}
\DeclareSymbolFont{matha}{U}{matha}{m}{n}

\DeclareMathSymbol{\varsubset}{3}{matha}{"80}

\usepackage{color}   
\usepackage[hypertexnames=false]{hyperref}
\hypersetup{
    linktoc=all,     
    linkcolor=black,  
}

\renewcommand{\i}{\infty}
\renewcommand{\o}{\otimes}

\usepackage{epigraph}

\begin{document}

\begin{abstract}
We introduce and study the non-connective spectral stack $\M$, the moduli stack of oriented formal groups.
We realize some results of chromatic homotopy theory in terms of the geometry of this stack. For instance, we show that
 its descent spectral sequence recovers the Adams-Novikov spectral sequence. For two $\E$-forms of periodic complex bordism $\MP$, the Thom spectrum and Snaith construction model, we describe the universal property of the cover $\Spec(\MP)\to\M$. 
We show that Quillen's celebrated theorem on complex bordism is equivalent to the assertion that the underlying ordinary stack of $\M$ is the classical stack of ordinary formal groups $\mathcal M^\heart_\mathrm{FG}$.
In order to carry out all of the above, we develop foundations of a functor of points approach to non-connective spectral algebraic geometry.
\end{abstract}

\maketitle

\section*{Introduction}

The goal of this paper is to show that the connection between the homotopy category of spectra and formal groups, which is at the heart of chromatic homotopy theory, may be manifested in terms of spectral algebraic geometry.
By \textit{formal group}, we mean a  smooth $1$-dimensional commutative formal group, as is traditional in homotopy theory. These are classified by an algebro-geometric moduli stack, the \textit{moduli stack of formal groups}, which we shall denote $\mathcal M^\heart_\mathrm{FG}$.

Here and in the rest of the paper, we will use the heart symbol $\heart$ in the  subscript to emphasize the classical nature of algebro-geometric objects considered,  in contrast to spectral algebro-geometric objects which will be our main focus. This notation is motivated by the fact that ordinary abelian groups sit inside spectra as the heart of the usual  $t$-structure.
For instance, $\mathcal M^\heart_\mathrm{FG}$ is our notation for the classical moduli stack of formal groups. The superscript $\heart$ is part of the notation here, reminding of us that we are considering it as an object of ordinary algebraic geometry.

The relationship between the moduli stack $\mathcal M^\heart_\mathrm{FG}$ and the stable homotopy category is well known, and has been one of the main organizing principles in homotopy theory for the last 50 years. Let us quickly review a variant of the usual approach to it, with the only non-standard aspect being our focus on $2$-periodic ring spectra (see Remark \ref{periodica}), in the following sense:

\begin{definnition}[{\cite[Definition 4.1.5]{Elliptic 2}}]
A commutative ring spectrum $A$ is said to be \textit{weakly $2$-periodic} if the ring structure induces an isomorphism $\pi_{*+2}(A)\simeq \pi_*(A)\o_{\pi_0(A)}\pi_2(A)$.
\end{definnition}

 The starting point is the observation that any complex orientation on a weakly $2$-periodic periodic commutative ring spectrum $A$ gives rise through $A^0(\mathbf{CP}^\i)$ to a formal group law  over $\pi_0(A)$. The universal weakly $2$-periodic complex orientable commutative ring is the \textit{periodic complex bordism spectrum} $\MP:=\bigoplus_{i\in\mathbf Z}\Sigma^{2i}(\mathrm{MU})$. The following celebrated result of Quillen may be seen as the origin of chromatic homotopy theory:

\begin{theoremm}[Quillen, {\cite{Quillen}}]
There is a canonical isomorphism of graded commutative rings $\pi_0(\mathrm{MP})\simeq L$, between the homotopy ring of the periodic complex bordism spectrum $\mathrm{MP}$, and the \textit{Lazard ring} $L$, classifying formal group laws. Under this isomorphism, the formal group law $\mathrm{MP}^0(\mathbf{CP}^\infty)$ corresponds to the universal formal group law.
\end{theoremm}

Formal group laws give rise to formal groups, and since the Lazard ring $L$ classifies formal groups, this produces a map
$\Spec(L)\to \mathcal M^\heart_\mathrm{FG}$ to the moduli stack of formal groups. The fact that any formal group can be locally presented by a formal group law, implies that $\Spec(L)\to\Mo$ is a faithfully flat cover. Quillen's Theorem may be reformulated in terms of this cover, as asserting that periodic complex bordism exhibits a groupoid presentation for the moduli stack of formal groups as
$$
\Mo\simeq \varinjlim\big(\Spec(\pi_0(\mathrm{MP}\o\mathrm{MP})\rightrightarrows \Spec(\pi_0(\mathrm{MP}))\big).
$$
One crucial consequence of this is that the second page of the \textit{Adams-Novikov spectral sequence}
$$
E^{s,t}_2 = \mathrm{Ext}^{s,t}_{\pi_0(\MP\o\MP)} (\pi_*(\MP), \pi_*(\MP))\Rightarrow \pi_{t-s}(S)
$$
may be determined purely in terms of algebro-geometric formal group data as
\begin{equation}\label{second page}
E^{s, 2t}_2 \simeq \mathrm H^s(\mathcal M^\heart_\mathrm{FG}; \omega_{\mathcal M^\heart_\mathrm{FG}}^{\otimes t}).
\end{equation}
Here $\omega_{\mathcal M^\heart_\mathrm{FG}}$ is the quasi-coherent sheaf on $\Mo$, collecting the modules of invariant Kähler differentials on formal groups.

The Adams-Novikov spectral sequence is one of our most powerful computational tools for studying the stable homotopy groups of spheres. Following the insights of Jack Morava, much of the structure of the stable stem uncovered this way has since been explained, using the isomorphism \eqref{second page}, in terms of the structure of the moduli stack $\mathcal M_\mathrm{FG}^\heartsuit$. More precisely, this was traditionally  viewed in terms of formal groups laws instead of the stack of formal groups $\mathcal M_\mathrm{FG}^\heartsuit$. The  stacky perspective on chromatic homotopy theory, which may be seen as a natural outgrowth of  Morava's orbit picture \cite[Chapter 4]{Orange Book},  was first explicitly emphasized and disseminated by Mike Hopkins in \cite{COCTALOS}. For a more thorough account of chromatic homotopy theory from this perspective, see \cite{Lurie Chromatic}.

\begin{remmark}\label{periodica}
Most of the literature on chromatic homotopy theory, e.g.~\cite{Green Book} and \cite{Orange Book}, work with complex oriented ring spectra, without the weakly  2-periodicity assumption we imposed above. In particular, both Quillen's Theorem and the Adams-Novikov spectral sequence are usually stated in terms of the complex bordism spectrum $\mathrm{MU}$, instead of its periodic variant $\mathrm{MP}$. However, for instance by \cite[Example 7.4]{Goerss on ANSS}, the Adams-Novikov spectral sequence is the same whether formed in terms of $\mathrm{MU}$ or in terms of $\mathrm{MP}$.
\end{remmark}

In this paper, we show how the above-discussed relationship between the moduli stack of formal groups $\mathcal M^\heart_\mathrm{FG}$ and spectra, may be encoded in terms of spectral algebraic geometry. That is a variant of usual algebraic geometry, in which the basic role of ordinary commutative rings as affines is instead replaced by $\E$-ring spectra (called just $\E$-rings for short). These homotopy-coherent analogues of commutative rings, first introduced in \cite{MQR}, have since been shown to accommodate analogues of much of usual commutative algebra, e.g.~\cite{EKMM}. For an account from an $\i$-categorical perspective, see for instance \cite{HA}. Spectral algebraic geometry was first introduced in \cite[Chapter 2.4]{HAG2}, its foundations laid out extensively  in \cite{SAG}, and it has since found spectacular applications in \cite{Elliptic 2} by providing  elegant approaches to topological modular forms and Lubin-Tate spectra.

Unfortunately, the theory developed in \cite{SAG} works under a connectivity assumption, or alternatively without it, but then restricting itself to  Deligne-Mumford stacks. Those are roughly the stacks which admit an \'etale cover by schemes, and provide a sufficient context for many applications. This is not the case when dealing with the moduli stack of formal groups, or spectral enhancements thereof, since it only admits a flat cover by schemes - see Remark \ref{Why fpqc}. It similarly does not suffice for us to restrict to the connective setting, since our main spectral stacks of interest $\M$, defined below, will be inherently non-connective. For this reason,  we spend Section \ref{Chapter 1} of this paper developing an appropriate setting of non-connective spectral algebraic geometry for our desired applications.

We follow a functor of points approach, considering a \textit{non-connective spectral stack} $X$ to be a certain kind of functor $X:\CAlg\to\mS$, going from the $\i$-category of (not-necessarily-connective) $\E$-rings $\CAlg$, and into the  $\i$-category of spaces $\mS$.
To every non-connective spectral stack $X$, we associate an \textit{underlying ordinary stack} $X^\heart : \CAlg^\heart\to\mS$, a (co)presheaf on the category of ordinary commutative rings $\CAlg^\heart$, as described in the next construction.

\begin{conns}\label{build underlying}
Write the non-connective spectral stack $X$  as a colimit 
$
X\simeq \varinjlim_i \Spec (A_i)
$
in the $\i$-category of non-connective spectral schemes of some diagram
of affine non-connective spectral schemes $\Spec(A_i)$, i.e.\ functors representable by $\E$-rings $A_i$. Then its underlying ordinary stack is defined to be the colimit of the corresponding diagram of ordinary affine schemes
$$
X^\heart := \varinjlim_i \Spec(\pi_0(A_i)),
$$
with the colimit computed in the $\i$-category of ordinary stacks.
\end{conns}

In the connective case discussed in \cite{SAG}, the underlying spectral stack $X^\heart$ of a spectral stack $X$ coincides with the functor restriction $X\vert_{\CAlg^\heart}$. That  no longer holds in the non-connective world, see Remark \ref{Heart is fickle}, making the notion more subtle.

\begin{remmark}
Throughout this paper, we adopt the terminology \textit{ordinary stack} to refer to certain kinds of functors $\CAlg^\heart \to \mS$  (essentially the higher stacks of \cite{Simpson} or \cite{Toen}) from the category of ordinary commutative rings $\CAlg^\heart$. These are objects of ordinary, as opposed to spectral, algebraic geometry, in the sense that they have nothing to do with $\E$-rings. On the other hand, they do take values in the $\i$-category of spaces $\mS$, while (the functors of points of) algebraic stacks in the algebro-geometric literature are mostly required to being valued in the $2$-category of groupoids $\mathcal G\mathrm{rpd}$. In light of the standard equivalence $\tau_{\le 1}(\mS)\simeq \mathcal G\mathrm{rpd}$ between $1$-truncated spaces and groupoids, our notion is a pure generalization - so a groupoid-valued functor like $\Mo$ is indeed a special case.
\end{remmark}

Similarly to the above construction of the underlying ordinary stack, any non-connective spectral stack $X$ admits an $\i$-category of quasi-coherent sheaves $\QCoh(X)$ by the limit formula
$$
\QCoh\big(\varinjlim_i\Spec(A_i)\big)\simeq \varprojlim_i\Mod_A,
$$
where $\Mod_A$ denotes the $\i$-category of $A$-module spectra. To be precise, this $\QCoh(X)$ is not analogous to the usual abelian category of quasi-coherent sheaves on an ordinary stack, but rather to the derived category of quasi-coherent sheaves. This kind of abuse of language permeates the transition from ordinary to derived or spectral algebraic geometry.
Under the assumption that the non-connective spectral stack $X$ is \textit{geometric} - roughly:  it admits a flat cover by affines - a quasi-coherent sheaf $\sF$ on $X$ in this sense induces homotopy groups $\pi_t(\sF)$, which are quasi-coherent sheaves in the usual sense on the underlying ordinary stack $X^\heart$. They are related to the spectrum $\Gamma(X; \sF)$ of global sections of $\sF$ by the \textit{descent spectral sequence} 
\begin{equation}\label{DSS for geometric}
E^{s,t}_2 = \mathrm H^s(X^\heart; \pi_t(\sF))\Rightarrow \pi_{t-s}(\Gamma(X; \sF)).
\end{equation}

This concludes the outline of non-connective spectral algebraic geometry, the details of which we develop Section \ref{Chapter 1}. We now move on to outlining our applications of it in Section \ref{Section 2} to the connection between formal groups and spectra.

Let us start with some motivation.
In his work on topological modular forms, sketched in \cite{survey} and fleshed out in \cite{Elliptic 2}, one of Lurie's central ideas was to define a moduli stack $\mathcal M^\mathrm{or}_\mathrm{Ell}$ purely inside the realm of spectral algebraic geometry, parameterizing certain spectral analogues of elliptic curves. The Goerss-Hopkins-Miller Theorem on the $\E$-structure on $\mathrm{TMF}$  then reduces to a comparison between this spectral stack $\mathcal M^\mathrm{or}_\mathrm{Ell}$, defined by a universal property, and the ordinary stack of elliptic curves $\mathcal M^\heart_\mathrm{Ell}$ in the usual sense.
Lurie also proved in \cite{Elliptic 2} that an analogue of this works to define Lubin-Tate spectra, also known as Morava E-theory, or completed Johnson-Wilson theory,  together with their $\E$-rings structures.

We show that a similar approach works when elliptic curves are replaced with formal groups. There does indeed exist a well-behaved theory of formal groups in spectral algebraic geometry, but in order to obtain the correct moduli stack, that is not sufficient. Following Lurie, we must impose the following somewhat exotic additional assumption, with no  direct classical analogue.

\begin{definnition}[{\cite[Definition 4.3.9]{Elliptic 2}},  Definition \ref{second def of or}]\label{Def of or}
An \textit{orientation} on a formal group $\w{\G}$ over an $\E$-ring $A$ is an element of $\pi_2(\w{\G}(\tau_{\ge 0}(A))$, which induces an equivalence $\omega_{\w{\G}}\simeq \Sigma^{-2}(A)$ of $A$-module spectra.
\end{definnition}

The notion of an orientation of a formal group, which might more accurately be called a $2$-shifted orientation, may seem somewhat arbitrary, but see Remark \ref{orientation motivation} for some motivation.
In terms of this notion, define our central  object of interest $\mathcal M^\mathrm{or}_\mathrm{FG}$ to be the \textit{moduli stack of oriented formal groups}. That is to say, it is  the functor $\M:\CAlg\to\mS$, sending any $\E$-ring to the space of oriented formal groups over it. This space turns out to be 
particularly simple: it is contractible if and only if the $\E$-ring $A$ is complex periodic (both weakly $2$-periodic periodic and complex orientable), and empty otherwise. This is the content of \cite[Proposition 4.3.23]{Elliptic 2}, summarized here as Proposition \ref{Paying proper credit}, and is some sense the key technical observation underlying this paper.

\begin{remmark}
Though the notion of an oriented formal group is due to Lurie, their moduli stack $\M$ does not explicitly appear in his work. Related stacks, such as the moduli of oriented elliptic curves and oriented deformation spaces, are considered and studied in depth in \cite{Elliptic 2}. We summarize their relationship to $\M$ in Example \ref{What came before}. Therefore, although our analysis of it is based firmly on the Lurie's foundational work, the consideration of the stack $\M$ is a novelty of the present work.
\end{remmark}

\begin{motto}
The algebraic geometry of the non-connective spectral stack $\M$ manifests chromatic homotopy theory.
\end{motto}

The following is our main result, making precise in what way the non-connective spectral stack $\M$ captures the relationship between formal groups and spectra.

\begin{theoremm}[Theorem \ref{M is geometric},  Corollary \ref{Underlying ordinary stack is MFG}, Proposition \ref{functions are S}, Proposition \ref{ANSS}, and Theorem \ref{IndCoh is Sp} ]\label{Maintheoremintrod}
The following statements hold for the moduli stack of oriented formal groups $\M,$ as defined above:
\begin{enumerate}[label = (\roman*)]
\item It is a geometric non-connective spectral stack. \label{mtun}
\item Its underlying ordinary stack is canonical equivalent to\label{mtdeux}
the classical moduli stack of formal groups $\mathcal M^\heart_\mathrm{FG}$.
 \item There is a canonical equivalence of $\E$-rings\label{mttrois}
$
\sO(\M)\simeq S
$
between the $\E$-ring of global functions on $\M$ and the sphere spectrum.
\item The descent spectral sequence \label{mtquatre}
$$
E^{s,t}_2 = \mathrm H^s(\mathcal M_\mathrm{FG}^\heart; \pi_t(\sO_{\M}))\Rightarrow \pi_{t-s}(\sO(\M)).
$$
recovers the Adams-Novikov spectral sequence.
\item Under the definition of ind-coherent sheaves as in Definition \ref{Def of Coh}, there is a canonical \label{mtcinq} equivalence of symmetric monoidal $\i$-categories
\begin{eqnarray*}
\IndCoh(\M) &\simeq& \Sp\\
\sF &\mapsto &\Gamma(\M; \sF),
\end{eqnarray*}
given in terms of global sections of ind-coherent sheaves on $\M$.
\end{enumerate}
\end{theoremm}

\begin{proof}[Sketch of proof]  Choosing an $\E$-ring structure on the periodic complex bordism spectrum $\MP$ gives rise to a morphism $\Spec(\MP)\to\M$. We first show that this is a flat cover, in particular verifying \ref{mtun}.  The \v{C}ech nerve of this cover provides a simplicial presentation
$$
\xymatrix{
\M \,\, \simeq \,\, \varinjlim\Big( \cdots  \ar@<-1.5ex>[r]\ar@<-.5ex>[r] \ar@<.5ex>[r] \ar@<1.5ex>[r]& \Spec(\MP\o\MP\o\MP)\ar@<-1ex>[r]\ar[r] \ar@<1ex>[r] &\Spec(\MP\o\MP)\ar@<-.5ex>[r]\ar@<.5ex>[r] & \Spec(\MP)\Big ).
}
$$
In light of this presentation, Quillen's Theorem implies the underlying stack statement \ref{mtdeux}. Deducing  \ref{mttrois} boils down to well-known facts about nilpotent completion: namely that the sphere spectrum $S$ equals its own nilpotent completion with respect to the complex bordism spectrum, as first proved by Bousfield \cite[Theorem 6.5]{Bousfield}.
The assertion \ref{mtcinq} is a simple conseuqnce of \ref{mttrois}, in light of the definition of ind-coherent sheaves we are using. It remains only to verify \ref{mtquatre}, i.e.\ to identify the descent spectral sequence and the Adams-Novikov spectral sequences.  Again through the simplicial presentation, the descent spectral sequence in question is identified with the spectral sequence associated to a cosimplicial spectrum. That cosimplicial spectrum  is $\MP^{\otimes \bullet+1}$, often called the cobar resolution or the Amitsur complex. This implies  \ref{mtquatre}, since the standard construction Admas-Novikov spectral sequence is as the spectral sequence of that same cosimplicial spectrum, see \cite[Proposition 2.14]{MNN}.
\end{proof}

Theorem \ref{Maintheoremintrod} may be viewed as one
solution to the realization problem posed 
 in \cite{Goerss Realizing}. It asks roughly whether it is possible to enhance the classical stack of formal groups $\Mo$ to a spectral stack, whose descent spectral sequence would recover the Adams-Novikov spectral sequence. The setting for non-connective spectral algebraic geometry in \textit{loc.~cit.} is however slightly different from the ones we use in this paper. In that context, one is looking for an appropriate sheaf of $\E$-rings on the underlying ordinary stack, which then serves as the spectral structure sheaf.  Instead of such a vaguely \textit{ringed space} approach to non-connective spectral stacks, we tackle them exclusively from a \textit{functor of points} perspective. As such, the relationship with their underlying ordinary stacks, while still sufficiently workable, is somewhat less transparent.

On the other hand,
the fact that the Adams-Novikov spectral sequence may be interpreted in this way in terms of the simplicially-given non-connective spectral stack given above, is not new. For instance, it is mentioned explicitly in \cite[Remark 9.3.1.9]{SAG}. The novelty of our approach is, from this perspective, instead in describing the non-connective spectral stack $\M$ directly as a moduli stack, rather than  defining it in an \textit{ad hoc} way in terms of the complex bordism spectrum.

In the last section of the paper, Section \ref{Section 3}, we turn to the connection between the moduli stack of oriented formal groups $\M$ and periodic complex bordism $\MP$. When neglecting the $\E$-ring structure, the underlying spectrum of  $\MP$ admits the following description:

\begin{corollarry}[Remark \ref{MP as a homology theory}]
The extraordinary 
homology theory of periodic complex bordism $\MP_n :\mS\to\mathrm{Ab}$ may be written for any space $X$ as
$$
\MP_n(X)
\simeq
\Gamma\big(\mathcal M^\heart_\mathrm{FG}; \pi_n(\sO_{\M}[X] )\otimes_{\sO_{\mathcal M^\heart_\mathrm{FG}}}
\sO_{\mathcal M^{\heart,\mathrm{coord}}_\mathrm{FG}}\big),
$$
where $\mathcal M^{\heart, \mathrm{coord}}_\mathrm{FG}$ is the moduli stack of coordinatized formal groups, e.g.\ formal group laws.
\end{corollarry}

The situation becomes much subtler when trying to take an $\E$-structure into account. By \cite{Hahn-Yuan}, the spectrum $\MP$ (or even, $\mathbb E_2$-ring spectrum) admits at least two distinct non-equivalent $\E$-ring structures:
\begin{itemize}
\item The \textit{Thom spectum $\E$-ring} $\mathrm{MUP}$ is the Thom spectrum of the $J$-homomorphism $J :\mathbf Z\times\mathrm{BU}\to\Sp$, obtained by group completion from the symmetric monoidal functor $\mathrm{Vect}^\simeq_{\mathbf C}\to \Sp$, which sends a finite dimensional complex vector space $V$ to the suspension spectrum of its one-point compacification $S^V=V\cup\{\infty\}$.
\item The \textit{Snaith construction} $\MP_\mathrm{Snaith} := (S[\mathrm{BU}])[\beta^{-1}]$ is obtained from the suspension $\E$-ring $S[\mathrm{BU}]$ of the $\E$-space $\mathrm{BU}$, by inverting the canonical Bott element $\beta\in \pi_2(S[\mathrm{BU}])$, arising from the inclusion $S^2\simeq \mathbf {CP}^1\subseteq\mathbf{CP}^\i\simeq \mathrm{BU}(1)\subseteq \mathrm{BU}.$
\end{itemize}
By the main theorem of \cite{Hahn-Yuan}, these two are not equivalent as $\E$-rings. If we take $\MP$ to mean either of these, we obtain as discussed above a flat cover $\Spec(\MP)\to\M$. In light of it, we can explicitly describe the spectral algebro-geometric moduli problem  which periodic complex bordism corresponds to.

\begin{theoremm}[Theorem \ref{What MUP is} and Proposition \ref{What Snaith is}]
The affine non-connective spectral scheme $\Spec(\MP)$ parametrizes oriented formal groups $\w{\G}$ over $\E$-rings $A$, together with
\begin{itemize}
\item For  $\MP\simeq \mathrm{MUP}$:  a system of  trivializations $\omega^{\otimes V}_{\w{\G}}\simeq A$ of $A$-module spectra, functorial and symmetric monoidal in terms of finite-dimensional complex vector spaces $V$ and $\mathbf C$-linear isomorphisms.
\item {For  $\MP\simeq \MP_\mathrm{Snaith}$}: an $\E$-ring map $S[\mathrm{BU}]\to A$, which induces an equivalence of $A$-module spectra $\omega_{\w{\G}}\simeq A$.
\end{itemize}
\end{theoremm}

\begin{remmark}
Both of the universal properties for $\Spec(\MP)$ above have to do with oriented formal groups $\w{\G}$  over $\E$-rings $A$, equipped with an $A$-linear trivialization $\omega_{\w{\G}}\simeq A$ of the dualizing line $\omega_{\w{\G}}$, together with some built-in unitary group equivariance. The main difference in that in the case of the Thom spectrum $\E$-ring structure, the unitary group action ultimately stems from the usual action of $\mathrm U(n)$ on $\mathbf C^n$, while in the case of the Snaith construction, the trivial $\mathrm U$-action plays a role instead.
\end{remmark}

Finally we turn our attention to Quillen's Theorem itself, isolating its spectral algebro-geometric content as follows:

\begin{theoremm}[Theorem \ref{Quillen Rephrased}]
Quillen's Theorem is equivalent to assertion \ref{mtdeux} of Theorem \ref{Maintheoremintrod}, which is to say that there exists
an equivalence of ordinary stacks
$$
(\M)^\heart\simeq \mathcal M^\heart_\mathrm{FG}
$$
between the underlying ordinary stack $(\M)^\heart$ of the spectral moduli stack of oriented formal groups $\M,$ and the ordinary stack of formal groups $\mathcal M^\heart_\mathrm{FG}$.
\end{theoremm}

One striking aspect of this equivalent rephrasing of Quillen's Theorem, is that complex bordism makes no appearance in it. Indeed, our definition of the moduli stack of oriented formal groups $\M$ is completely independent of the spectrum $\MP$, whose role is instead only to provide a convenient flat cover of it. The proofs of most fundamental properties of the non-connective spectral stack $\M$ do still rely on complex bordism, so it and Quillen's Theorem are still essential, but perhaps their geometric role is clarified somewhat.

It would be interesting to have a direct algebro-geometric proof of the underlying ordinary stack result of the previous Theorem, which would therefore constitute a new proof of Quillen's Theorem. We hope to return to this problem in future work.
Another topic we will return to in a  follow-up \cite{ChromaticFiltration} to this paper, is the relationship  between chromatic localizations of spectra and the algebraic geometry of the moduli stack of oriented formal groups $\M$. There we will also clarify the connection between \cite{DevHop} and the present paper.  Let us point out, however, that our work is no way alone in expanding upon the ideas from \cite{Elliptic 2} for chromatic applications; see for instance \cite{Davies1}, \cite{Davies2}, and \cite{Sanath}.

\subsection*{Relationship to prior work}

The results of this paper constitute validation of an intuitive picture about how the field of chromatic homotopy theory should relate to spectral algebraic geometry. This picture is folklore in the field: it is implicit in \cite{COCTALOS}, indicated in \cite{Lurie Chromatic}, \cite[Section  2.4]{HAG2},  and made center-stage in \cite{PetersonBook}; but prior to this work, it had never been made rigorous in full generality. Perhaps the first explicit connection to spectral algebraic geometry, though from the perspective of sheaves of $\E$-rings, is in the work of Goerss-Hopkins-Miller on topological modular forms, as surveyed in \cite[Chapter 12]{TMF} or \cite{Goerss Bourbaki}.
The field of spectral algebraic geometry was largely set up, in the foundational monograph \cite{SAG} (and relying further on the  groundwork in $\i$-category theory and higher algebra \cite{HTT} and \cite{HA} respectively) to clarify this relationship. 
The precise connection to spectral algebraic geometry in the  deformation-theoretic and elliptic-curve settings is therefore due to the work of Lurie, as envisioned in the survey \cite{survey}, and carried out in  the series of papers \cite{Elliptic 1}, \cite{Elliptic 2}, \cite{Elliptic 3}. Others extended it in several directions:  Hill-Lawson to elliptic curves with level structure in \cite{Hill-Lawson},  Behrens-Lawson to certain higher-dimensional Shimura varieties \cite{TAF},  Davies to $p$-divisible groups in \cite{Davies1}, and more.

The role of this paper is to extend Lurie \textit{et.al.}'s work from the elliptic curve and formal deformation cases, to the entirety of formal groups (though certain remarks \cite[Example 9.3.1.8, Remark  9.3.1.9]{SAG} suggest this  might have been known to Lurie).
As discussed in \cite{Goerss Realizing} and \cite{Goerss Bourbaki}, the idea of extending the spectral-algebro-geometric story to the entirety of the stack of formal groups had been considered previously, but proved challenging to tackle using the contemporary obstruction-theoretic ringed-space approach. Other than building on the ideas of Lurie from \cite{Elliptic 2}, our technical innovation is to instead use the `functor of points' approach to non-connective spectral algebraic geometry (in a setting not developed in \cite{SAG}) to circumvent those difficulties.
The present work (as well as its sequel \cite{ChromaticFiltration}) therefore, by way of realizing classical chromatic results in spectral-algebro-geometric terms, provides  verification of a previously-elusive conjectural relationship between chromatic homotopy and spectral algebraic geometry.

\subsection*{Acknowledgments}
I am grateful to  David Ben-Zvi, John Berman, Andrew Blumberg, and Paul Goerss, for encouraging the work on, paying interest to, and giving valuable feedback on this project.  Equally  indispensable was the immense generosity of Sam Raskin, who lent his help with a number of technical details. Thanks to Anna Marie Bohman and Hiro Lee Tanaka for useful suggestions regarding certain non-technical aspects of the presentation, and to Niko Naumann for pointing out an mistake in a proof in a previous  version of this paper.
Special thanks go to John Berman, Adrian Clough, Desmond Coles, Tom Gannon, Saad Slaoui, and Jackson Van Dyke, for participating in the \textit{Chromatic Homotopy Theory Learning Seminar} at UT Austin in the Spring semester of 2021, without the stimulating atmosphere of which this paper would have never come about.

\section{Non-connective spectral algebraic geometry}\label{Chapter 1}
The first half of this paper is spent setting up sufficient foundations of non-connective spectral algebraic geometry, approached from the ``functor of points'' perspective, in order to be able to carry out the arguments we wish to in the second half.

\subsection{The fpqc topology}

Recall from \cite[Definition B.6.1.1]{SAG} that a map of (not necessarily connective) $\E$-rings $A\to B$ is \textit{faithfully flat} if
\begin{itemize}
\item the map of ordinary rings $\pi_0(A)\to\pi_0(B)$ is faithfully flat,
\item the map $\pi_*(A)\otimes_{\pi_0(A)}\pi_0(B)\to\pi_*(B)$ is an isomorphism of graded rings.
\end{itemize}
This allows us by \cite[Proposition B.6.1.3 and Variant B.6.1.7]{SAG} to equip the $\i$-categories $\CAlg^\mathrm{op},$ $(\CAlg^\mathrm{cn})^\mathrm{op},$ and $(\CAlg^\heart)^\mathrm{op}$ with the \textit{fpqc\footnote{This stands for \textit{fid\`elment-plate et quasi-compacte}, French for faithfully flat and quasi-compact.} topology}, giving rise to the corresponding  $\i$-topoi $\mathcal S\mathrm{hv}_\mathrm{fpqc}^\mathrm{nc} := \mathcal S\mathrm{hv}(\CAlg^\mathrm{op})$, $\mathcal S\mathrm{hv}_\mathrm{fpqc} :=\mathcal S\mathrm{hv}((\CAlg^\mathrm{cn})^\mathrm{op})$, and $\Shv^\heart_\mathrm{fpqc} :=\Shv((\CAlg^\heart)^\mathrm{op})$ respectively.

\begin{ex}
Following the standard convention in functor-of-points algebraic geometry, we call representables \textit{affine}. An \textit{affine non-connective spectral scheme} is therefore a functor $\Spec(A) : \mathrm{CAlg}\to\mS$, given by $R\mapsto\Map_{\CAlg}(A, R)$, for some $\E$-ring $A$. Similarly, an \textit{affine spectral scheme} is the functor $\Spec(A): \CAlg^\mathrm{cn}\to\mS$ given by $R\mapsto \Map_{\CAlg^\mathrm{cn}}(A,R)$ for a connective $\E$-ring $A$, and an \textit{affine ordinary scheme} if a functor $\Spec(A) : \CAlg^\heart\to\Set\hookrightarrow\mS$ given by $R\mapsto \Hom_{\CAlg^\heart}(A, R)$.
\end{ex}

\begin{remark}
The notation $\Shv^\heart_\mathrm{fpqc}$ is potentially misleading. It refers to the full subcategory of $\Fun(\CAlg^\heart, \mS)$, spanned by functors satisfying fpqc descent, and not as one might think the underlying ordinary topos $\tau_{\le 0}(\Shv_\mathrm{fpqc})\subseteq \Fun(\CAlg^\heart, \Set)$ of the $\i$-topos $\Shv_\mathrm{fpqc}.$ In the language of Simpson or Toën, see e.g. \cite{Toen}, we might say that $\mathrm{Shv}^\heart_\mathrm{fpqc}$ encodes higher (but not derived or spectral) stacks.
\end{remark}

 \begin{remark}\label{Warning: set-theory}
 There are some prickly point-set theoretic difficulties concerning the fpqc topology, making certain aspects of the theory, such as sheafification, require extra care to define correctly. We side-step these difficulties by choosing a Grothendieck universe and implicitly working inside its confines throughout. This kind of restriction would in general be unreasonable, as it is far too restrictive for a satisfactory general theory. Alas, we are in this paper ultimately only concerned with a small number of very specific examples of fpqc sheaves, all of which arise under small colimits from representables, making such stringent restrictions acceptable. We will mostly leave set-theoretic assumptions implicit, but see Remark \ref{Warning 2} for the one instance where some care is in fact required.
 \end{remark}

\begin{remark}\label{Why fpqc}
In light of the difficulties alluded to the previous remark, let us explain why we nonetheless use the fpqc topology. In Lemma \ref{Quillen} below, we will find ourselves (implicitly) faced with a ring map of the form $R\to R[t_1, t_2, \ldots]$, wishing to assert that it is a cover for a Grothendieck topology on commutative rings. Unfortunately, this map is neither an open immersion, \'etale, nor finitely presented, excluding from the list of candidates the Zariski, \'etale, and fppf topologies, or any in between. The map in question is faithfully flat however, therefore making the fpqc topology the only one applicable among the standard Grothendieck topologies on commutative rings. And while the fpqc topology may seem to be too big at first glance, its excellent descent properties make it a perfectly reasonable setting nonetheless.
\end{remark}
 
 \subsection{Connective covers and underlying ordinary stacks}
 The three $\i$-topoi $\Shv^\mathrm{nc}_\mathrm{fpqc}$, $\Shv_\mathrm{fpqc}$, and $\Shv^\heart_\mathrm{fpqc}$ are related to each other by a series of adjunctions:
 
 \begin{prop}\label{sheaf adjunction}
 There are canonical adjunctions
 $$
 \tau_{\ge 0}:\Shv_\mathrm{fpqc}^\mathrm{nc}\rightleftarrows \Shv_\mathrm{fpqc},\qquad \Shv^\heart_\mathrm{fpqc}\rightleftarrows \Shv_\mathrm{fpqc} : (-)^\heart
 $$
with colimit-preserving  right adjoints. The unlabeled arrows in both  are fully faithful.
 \end{prop}
 
 \begin{proof}
The inclusion of connective $\E$-rings into all $\E$-rings and the formation of connective covers form an adjunction
$i:\CAlg^\mathrm{cn}\rightleftarrows \CAlg : \tau_{\ge 0}$. Restriction and left Kan extension along both of these functors induce adjunctions between presheaf $\i$-categories
$$
i_! :\Fun(\CAlg^\mathrm{cn}, \mS)\rightleftarrows \Fun(\CAlg, \mS) : i^*,\quad (\tau_{\ge 0})_!:\Fun(\CAlg, \mS)\rightleftarrows \Fun(\CAlg^\mathrm{cn}, \mS) : (\tau_{\ge 0})^*,
$$
and it follows from the adjunction between $i$ and $\tau_{\ge 0}$ by abstract nonsense that $i_!\simeq (\tau_{\ge 0})^*$.
Indeed, for any functor $X:\CAlg^\mathrm{cn}\to \mS$ and any $\E$-ring $A$, the colimit formula for Kan extension tells that
$
i_!(X)(A)\simeq \varinjlim_{B\in\CAlg^\mathrm{cn}\times_{\CAlg}\CAlg_{/A}}X(B).
$
Since the connective cover $\tau_{\ge 0}(A)\to A$ is the terminal object of the $\i$-category $\CAlg^\mathrm{cn}\times_{\CAlg}\CAlg_{/A}$, it follows that $i_!(X)(A)\simeq X(\tau_{\ge 0}(A))$ as desired.
From the equivalence $i_!\simeq (\tau_{\ge 0})^*$
it now follows that
$$i^*i_!\simeq i^*\tau_{\ge 0}^*\simeq (\tau_{\ge 0}\circ i)^!\simeq \mathrm{id},$$
showing that  $i_!$ is fully faithful.
 
On the other hand,  it is immediate from the definition of faithfully flat $\E$-ring maps that both of the functors $i$ and $\tau_{\ge 0}$ preserve them. Thus the restriction functors $i^*$ and $(\tau_{\ge 0})^* \simeq i_!$ both preserve the subcategories of sheaves, implying that the above presheaf adjunction restricts to the full subcategories of presheaves, giving rise to
$$
(\tau_{\ge 0})^*\simeq i_! :\Shv_\mathrm{fpqc}\rightleftarrows \Shv_\mathrm{fpqc}^\mathrm{nc} : i^*,
$$
whose left adjoint is still fully faithful. It also admits a further left adjoint in the form of $$\Shv_\mathrm{fpqc}^\mathrm{nc}\subseteq\Fun(\CAlg, \mS)\xrightarrow{(\tau_{\ge 0})_!}\Fun(\CAlg^\mathrm{cn}, \mS)\xrightarrow{L}\Shv_\mathrm{fpqc},
$$
where $L$ denotes sheafification. This is the functor $\tau_{\ge 0} : \Shv_\mathrm{fpqc}^\mathrm{nc}\to\Shv_\mathrm{fpqc}$ from the statement of the Proposition, and its right adjoint $(\tau_{\ge 0})^*\simeq i_!$ indeed preserves all colimits, and is fully faithful.

To obtain the other adjunction in the statement of the Proposition, we repeat the argument, starting with the adjunction $\pi_0:\CAlg^\mathrm{cn}\rightleftarrows \CAlg^\heart : j$. That gives rise to a colimit-preserving functor $(\pi_0)_!\simeq j^* :\Shv_\mathrm{fpqc}\to\Shv_\mathrm{fpqc}^\heart$ - this is the functor $(-)^\heart$ from the statement of the Proposition. It admits  a left adjoint given by
$$
\Shv_\mathrm{fpqc}^\heart\subseteq\Fun(\CAlg^\heart, \mS)\xrightarrow{j_!}\Fun(\CAlg^\mathrm{cn}, \mS)\xrightarrow{L}\Shv_\mathrm{fpqc},
$$
where $L$ is sheafification. It remains to show that this functor is fully faithful. That is, we must show that for any fpqc sheaf $X:\CAlg^\heart\to\mS$, the restriction along $j :\CAlg^\heart\subseteq\CAlg^\mathrm{cn}$ of the sheafification of the left Kan extension $j_!(X) :\CAlg^\mathrm{cn}\to\mS$ is equivalent to $X$ via the canonical map.

Recall from \cite[Subsection 6.5.3]{HTT} that for any functor $Y:\CAlg^\mathrm{cn}\to\mS$ (resp.\ $\CAlg^\heart\to\mS$), its sheafification may be obtained as a transfinite composite of the partial sheafification operation $Y\mapsto Y^+$, i.e. sequential colimit of the form
$$
L(Y)\simeq \varinjlim(Y\to Y^+ \to (Y^+)^+ \to ((Y^+)^+)^+ \to \cdots).
$$
Here the partial sheafification $Y^+$ of $Y$ is given for a connective $\E$-ring (resp.\ commutative ring) $A$ by
$$
Y^+(A)\simeq \varinjlim_{B\in \CAlg{}^\mathrm{ff}_A}\Tot(Y(B^{\otimes_A{\bullet+1}})),
$$
taking totalization of the evaluation of $Y$ on the \v{C}ech nerve of the covers, and
with the colimit indexed over all faithfully flat $\E$-ring (resp.\ comutative ring) maps $A\to B$. If $A$ is discrete, then it follows from the definition of flat maps of $\E$-rings that any faithfully flat cover will also be discrete, as will be all the relative smash products $B^{\otimes_A\bullet +1}$. Hence it follows from the above formulas that the canonical map $(Y\vert_{\CAlg^\heart})^+\to (Y^+)\vert_{\CAlg^\heart}$ is an equivalence, and by passage to transfinite composites, we conclude that $L(Y)\vert_{\CAlg^\heart}$ is precisely the fpqc sheafification of the functor $Y\vert_{\CAlg^\heart}$.

Returning to the case at hand, we are thus considering the sheafification of $j_!(X)\vert_{\CAlg^\heart}$. From the colimit formula for left Kan extensions, we see that for any $A\in\CAlg^\mathrm{cn}$
$$
(j_!(X))(A)\simeq \varinjlim_{B\in\CAlg^\heart\times_{\CAlg^\mathrm{cn}}\CAlg^\mathrm{cn}_{/A}}X(B).
$$
Of course when $A$ is  a discrete $\E$-ring, the overcategory $\CAlg^\heart\times_{\CAlg^\mathrm{cn}}\CAlg^\mathrm{cn}_{/A}$ has $A$ itself as a terminal object, showing that in that case $(j_!(X))(A)\simeq X(A)$ as desired.
 \end{proof}
 
 \begin{remark}\label{Warning 2}
 Proposition \ref{sheaf adjunction} is the one place in this paper where the issue of set-theoretic difficulties, discussed in Remark \ref{Warning: set-theory}, actually arise. We gave a naive treatment, ignoring set-theoretic difficulties. In reality, for fpqc sheafification to exist as we used it in the proof, we must allow ourselves to work in the not-necessarily-small fpqc topos, breaking our implicit smallness assumption established in Remark \ref{Warning: set-theory}. This is not much of an issue however, as we will only be concerned with fpqc sheaves which are obtained as small colimits of representables, in which case the fact that all the functors involved in the statement of Proposition \ref{sheaf adjunction} preserve colimits, ensures that the smallness requirement is in fact obeyed.
 \end{remark}
 
 In what follows, we will mostly suppress the fully faithful functors of Proposition \ref{sheaf adjunction}, and view  the $\i$-topos connective flat stacks (resp.~ordinary flat stacks) as a full subcategory $\Shv_\mathrm{fpqc}\subseteq \Shv_\mathrm{fpqc}^\mathrm{nc}$ (resp.\ $\Shv^\heart_\mathrm{fpqc}\subseteq \Shv_\mathrm{fpqc}$) of the $\i$-topos of non-connective flat stacks (resp.~connective flat stacks).
 We also  obtain a functor $\Shv_\mathrm{fpqc}^\mathrm{nc}\to\Shv_\mathrm{fpqc}^\heart$ given by $X\mapsto X^\heart := (\tau_{\ge 0}(X))^\heart$, whose restriction to the full subcategory $\Shv_\mathrm{fpqc}\subseteq\Shv_\mathrm{fpqc}^\mathrm{nc}$ recovers the eponymous functor $(-)^\heart$ appearing in Proposition \ref{sheaf adjunction}. We refer to $\tau_{\ge 0}(X)$ as the \textit{connective cover} of $X$, and to $X^\heart$ as its \textit{underlying ordinary stack}.

 \begin{remark}\label{connective cover through colimits} 
 We can read off from the proof of Proposition \ref{sheaf adjunction} that the functors $\tau_{\ge 0}:\Shv^\mathrm{nc}_\mathrm{fpqc}\to \Shv_\mathrm{fpqc}$ and $(-)^\heart :\Shv_\mathrm{fpqc}^\mathrm{nc}\to\Shv_\mathrm{fpqc}^\heart$ are given by left Kan extension (and sheafification), and may as such informally be written as
  $$\tau_{\ge 0}(\varinjlim_i \Spec(A_i))\simeq \varinjlim_i \Spec(\tau_{\ge 0}(A)),\qquad (\varinjlim_i \Spec(A_i))^\heart\simeq \varinjlim_i \Spec(\pi_0(A)).$$
In particular, the functors $\tau_{\ge 0}$ and $(-)^\heart$, as well as the inclusions $\Shv_{\mathrm{fpqc}}^\heart\subseteq\Shv_\mathrm{fpqc}\subseteq\Shv_\mathrm{fpqc}^\mathrm{nc}$, all preserve the respective notions of affine schemes.
\end{remark}

\begin{remark}
In contrast to the preceding example, not all ways to pass from non-connective to connective algebraic geometry preserve affines. For any $\E$-ring, the restriction $\Spec(A)\vert_{\CAlg^\mathrm{cn}}$ is an object of $\Shv_\mathrm{fpqc}$. Under the inclusion $\Shv_\mathrm{fpqc}\subseteq\Shv_\mathrm{fpqc}^\mathrm{nc}$, the functor $\CAlg\to \mS$ given by $R\mapsto\Map_{\CAlg}(A, \tau_{\ge 0}(R))$ is generally not affine unless $A$ itself had been connective to begin with. When $A$ is instead assumed to be coconnective, $\Spec(A)$ recovers the notion of \textit{coaffine stacks} of \cite[Section 4.4]{DAGVIII} (at least if we were working over a field of characteristic $0$), denoted there $\mathrm{cSpec}(A)$. These are closely related to \textit{affine stacks} of \cite{Champs Affines}, although the latter exist in a cosimplicial, rather than spectral setting. Note in particular that $\Spec(A)\not\simeq\mathrm{cSpec}(A)$ for a coconnective $\E$-ring $A$, i.e.\ coaffine stacks are not a special case of non-connective affine stacks.
\end{remark}

\begin{remark}\label{Undst for conct}
For a connective fpqc stack $X:\CAlg^\mathrm{cn}\to\mS$, its underlying ordinary fpqc stack $X^\heart$ admits another description as $X^\heart\simeq X\vert_{\CAlg^\heart}$. There is also in that case a canonical map $X^\heart\to X$ in $\Shv_\mathrm{fpqc}$, given by the unit of the relevant adjunction. Both of those fail for non-connective stacks; we will see in Remark \ref{Heart is fickle} an example of a non-connective fpqc stack $X$ for which $X^\heart$ is an interesting non-trivial stack, while $X\vert_{\CAlg^\heart}=\emptyset$ is the constant empty-set functor. Indeed, $X$ and $X^\heart$ are in general connected only through the cospan $X\to \tau_{\ge 0}(X)\leftarrow X^\heart$.
\end{remark}

\begin{exammple}
According to  \cite[Variant 1.1.2.9]{SAG},  non-connective spectral schemes, which embed fully faithfully into spectral stacks, may be described as follows. A non-connective spectral scheme $X$ consists of a $(|X|, \sO_X)$ of a topological space $|X|$ and a sheaf of $\E$-rings $\sO_X$ on it, satisfying a spectral analogue of the usual definition of a scheme in the locally ringed space approach. The underlying ordinary stack of $X$ then coincides with the usual scheme $(|X|, \pi_0(\sO_X))$. An analogous description works for spectral Deligne-Mumford stacks, if topological spaces are replaced with $\i$-topoi. The underlying ordinary stack construction should be thought of as a functor of points analogue of this.
\end{exammple}

\begin{remark}
The terminology of calling the $\tau_{\ge 0}(X)$ as the \textit{connective cover}  is taken to comply with the standard one for $\E$-rings, recovered for affine stacks. But note that the canonical map goes $X\to\tau_{\ge 0}(X)$, suggesting a name like ``connective quotient'' might be more appropriate.
\end{remark}

\subsection{Geometric stacks}
We introduce a convenient class of non-connective spectral stacks, a non-connective version of the geometric stacks of \cite[Definition 9.3.0.1]{SAG}. Just as Deligne-Mumford stacks are roughly those stacks which admit an \'etale cover, and Artin stacks are roughly those stacks which admit a smooth cover, so are geometric stacks roughly those stacks which admit a flat cover.

\begin{definition}\label{Def of geom stack}
A \textit{non-connective geometric stack} is a functor $X:\CAlg\to \mS$ which satisfies the following conditions:
\begin{enumerate}[label = (\alph*)]
\item The functor $X$ satisfies descent for the fpqc topology.
\item The diagonal map $X\to X\times X$ is affine.\label{Affine diagonal}
\item There exists an $\E$-ring $A\in\CAlg$ and a faithfully flat map $\Spec(A)\to X$.
\end{enumerate}
\end{definition}

\begin{remark}\label{Why affine diagonal?}
The affine diagonal condition \ref{Affine diagonal} above is as always  equivalent to the following statement (see \cite[Proposition  9.3.1.2]{SAG} for proof of the analogous result in the connective setting): for any pair of maps $\Spec(A)\to X$ and $\Spec(B)\to X$, the fiber product $\Spec(A)\times_X\Spec(B)$ is affine. This assumption is made largely to simplify various statements, and could be dropped in much of what follows, at the cost of demanding affine representability of various morphisms.
\end{remark}

\begin{variant}
By substituting $\CAlg\mapsto\CAlg^\mathrm{cn}$ in Definition \ref{Def of geom stack}, we recover the notion of a \textit{geometric stack} from \cite[Definition 9.3.0.1]{SAG}. It follows that a non-connective geometric stack that belongs to  $\Shv_\mathrm{fpqc}\subseteq\Shv_\mathrm{fpqc}^\mathrm{nc}$ is precisely a geometric stack. For another variant, we can substitute $\CAlg\mapsto\CAlg^\heart$ in Definition \ref{Def of geom stack} to obtain the notion of an \textit{ordinary geometric stack}.
\end{variant}

\begin{remark}
As discussed in  \cite[Subsection 9.1.6.]{SAG}, the inclusions $\Shv_\mathrm{fpqc}^\heart\to\Shv_\mathrm{fpqc},$ as well as its right adjoint $X\mapsto X^\heart$, both preserve the condition of a stack begin geometric.
\end{remark}

Geometric stacks in either of the three variants - non-connective, connective, and ordinary - admit presentations as quotients of flat affine groupoids in their respective setting:

\begin{prop}\label{Groupoid presentation}
A functor $X$ in $\mathcal S\mathrm{hv}^\mathrm{nc}_\mathrm{fpqc}$ is a non-connective geometric stack (resp.\ geometric stack, ordinary geometric stack) if and only if it can be written as a geometric realization $X\simeq |\Spec(A^\bullet)|$ of a groupoid object of the form $\Spec(A^\bullet)$ with $A^n\in\CAlg$ (resp.\ $\CAlg^\mathrm{cn}$, $\CAlg^\heart$) for all $n\ge 0$, and such that all of its  face maps are faithfully flat.
\end{prop}

\begin{proof}
For a non-connective geometric stack $X$, take a $\Spec(A^\bullet)$ to be the \v{C}ech nerve $\check{\mathrm{C}}^\bullet(\Spec(A)/X)$ of a faithfully flat map $\Spec(A)\to X$, whose existence is guaranteed by Definition \ref{Def of geom stack}. The induced map $|\Spec (A^\bullet)|\to X$ being an equivalence follows form the abstract nonsense of Lemma \ref{covtriv}.

For the converse direction, the proof of \cite[Corollary 9.3.1.4]{SAG} goes through in the non-connective and ordinary setting, since  the proof of crucial \cite[Lemma 9.3.1.1]{SAG} makes no use of the connectivity hypothesis.
\end{proof}

The following simple result could certainly be justified by invoking appropriate passages from \cite{HTT}, but we prefer to give a direct and straightforward proof instead.

\begin{lemma}\label{covtriv}
Let $X\to Y$ be a a map in a sheaf $\i$-topos $\mX = \mathcal S\mathrm{hv}(\mC)$, such that for every map $j(C)\to Y$ with $C\in \mC$, the fiber product $j(C)\times_X Y$ is representable by some $C'\in \mC$, and the induced morphism $C\to C'$ is a one-element cover in the topology on $\mC$. Then $X\to Y$ is an effective epimorphism, which is to say that the induced map $|\check{\mathrm C}^\bullet(X/Y)|\to Y$ is an equivalence in $\mX$.
\end{lemma}

\begin{proof}
Write $Y\simeq \varinjlim_i j(C_i)$. Then $X\times_Y j(C_i)\simeq j(C'_i),$ and the morphisms $C_i\to C'_i$ are covers in the topology on $\mC$ for every $i$. Thus $j(C_i)\to j(C'_i)$ is an effective epimorphism in $\mX$ essentially by the definition of $\i$-categorical sheaves, see \cite[Proposition 6.2.3.20]{HTT}, hence $|\check {\mathrm C}^\bullet(j(C_i')/j(C_i))|\to j(C_i)$ is an equivalence for every $i$. Furthermore, it follows from the definition of the \v{C}ech nerve that it satisfies base-change, so that the canonical map of simplicial objects
$$
\check{\mathrm C}^\bullet(X/Y)\times_Y j(C_i)\to \check{\mathrm C}^\bullet(j(C_i')/j(C_i))
$$
is an equivalence. Finally we use all of the discussed facts to exhibit the canonical map $|\check{\mathrm C}^\bullet(X/Y)|\to Y$ as a composition of equivalences in $\mX$
\begin{eqnarray*}
|\check{\mathrm C}^\bullet(X/Y)|&\simeq&|\check{\mathrm C}^\bullet(X/Y)|\times _Y \varinjlim_i j(C_i)\\
&\simeq &\varinjlim_i |\check {\mathrm C}^\bullet(X/Y)\times_Y j(C_i)| \\
&\simeq& \varinjlim_i |\check {\mathrm C}^\bullet(j(C'_i)/j(C_i))|\\
&\simeq & \varinjlim_i j(C_i)\simeq Y,
\end{eqnarray*}
where the one additional fact we used is that pullbacks commute with arbitrary colimits in an $\i$-topos.
\end{proof}

\begin{corollary}\label{Pres for trunc}
For any non-connective geometric stack $X$, the connective cover $\tau_{\ge 0}(X)$ is a geometric stack and the underlying ordinary stack $X^\heart$ is an ordinary geometric stack. Given a groupoid presentation  $X\simeq |\Spec (A^\bullet)|$ as in Proposition \ref{Groupoid presentation}, there are canonical equivalences $\tau_{\ge 0}(X) \simeq |\Spec(\tau_{\ge 0}(A^\bullet))|$ and $X^\heart\simeq |\Spec(\pi_0(A^\bullet))|$.
\end{corollary}

\begin{proof}
This follows from  Proposition \ref{Groupoid presentation} and Remark \ref{connective cover through colimits}.
\end{proof}

\subsection{Quasi-coherent sheaves}
Next, we turn our attention to quasi-coherent sheaves in non-connective spectral algebraic geometry.

 \begin{definition}\label{Def of QCoh}
The functor of \textit{quasi-coherent sheaves} $\QCoh :(\mathrm{Shv}_\mathrm{fpqc}^\mathrm{nc})^\mathrm{op}\to \CAlg(\PrL)$ is defined by right Kan extension from the association $A\mapsto \Mod_A$. That is to say, it is given by
$$
\QCoh\big(\varinjlim_i \Spec(A_i)\big)\simeq \varprojlim_i \Mod_{A_i}.
$$
 \end{definition}

 \begin{remark}
The above definition  works just as well for functors $\CAlg\to\mS$ that fail to satisfy fpqc descent, but that gains no extra generality since it is invariant under sheafification. Indeed, the proof of the analogous claim in the connective context \cite[Proposition  6.2.3.1]{SAG} does not use the connectivity hypothesis.
 \end{remark}
 
 \begin{remark}
As noted in \cite[Remark 6.2.2.2]{SAG}, the restriction of the functor $\QCoh$ onto the subcategory $\Shv_\mathrm{fpqc}\subseteq \Shv_\mathrm{fpqc}^\mathrm{nc}$ is equivalent  to the definition of quasi-coherent sheaves on a functor in \cite[Definition 6.2.2.1]{SAG}. This may for instance be seen by noting that it sends $\Spec(A)\mapsto \Spec(A)$ for any connective $\E$-ring $A$, and commutes with colimits by Proposition \ref{sheaf adjunction}, hence the Kan extension definition of the functor $\mathrm{QCoh}$ in both the connective and non-connective context implies that it is preserved under the subcategory inclusion $\Shv_\mathrm{fpqc}\subseteq\Shv_\mathrm{fpqc}^\mathrm{nc}$.
 \end{remark}
 
  Given a morphism $f : X\to Y$ in $\Shv_\mathrm{fpqc}^\mathrm{nc}$, we obtain from the definition of quasi-coherent sheaves an adjunction
$$
f^*:\QCoh(Y)\rightleftarrows \QCoh(X) : f_*,
$$
whose left adjoint $f^*$ we call \textit{pullback along $f$}, and whose right adjoint $f_*$ we call \textit{pushforward along $f$}.

\begin{ex}\label{Standard notation}
Let $p:X\to\Spec(S)$ be the terminal map. Under the equivalence $\QCoh(\Spec(S))\simeq \Sp$, we use for any $\sF\in\QCoh(X)$ and $M\in \Sp$ traditional notation
$$
\sO_X\otimes M :=p^*(M), \qquad \Gamma(X;\sF):=p^*(\sF),
$$
and the terminology \textit{global sections} and \textit{constant quasi-coherent sheaf} respectively. For $M=S$, we obtain the \textit{structure sheaf} $\sO_X :=p^*(S)$. Because quasi-coherent pullback is symmetric monoidal by construction, $\sO_X$ is the monoidal unit for the canonical symmetric monoidal operation $\o_{\sO_X}$ on $\QCoh(X)$. Its global sections $\sO(X):=\Gamma(X;\sO_X)$ are \textit{$\E$-ring of functions on $X$}, and chasing through the definitions shows it may be computed as
\begin{equation}\label{Compute functions}
\sO\big(\varinjlim_i\Spec(A_i)\big)\simeq \varprojlim_i A_i.
\end{equation}
 \end{ex}

\begin{remark}
The global functions functor $\sO: (\Shv_\mathrm{fpqc}^\mathrm{nc})^\mathrm{op}\to \CAlg$, introduced in Example \ref{Standard notation}, is right adjoint to the fully faithful embedding $\Spec :\CAlg\to (\Shv_\mathrm{fpqc}^\mathrm{nc})^\mathrm{op}$ discussed in Remark \ref{connective cover through colimits}. The limit preservations of right adjoints gives rise to the formula \eqref{Compute functions}. Even for a connective or ordinary fpqc stack $X$, the $\E$-ring $\sO(X)$ might still be a non-connective $\E$-ring. Indeed, for an ordinary stack, e.g.\ an ordinary scheme, the homotopy groups $\pi_{-i}(\sO(X))$ agree with quasi-coherent sheaf cohomology $\mathrm H^{i}(X; \sO_X)$ for all $i\in\Z$, which is often non-zero for various values of $i> 0$.
\end{remark}
 
\begin{remark}
The notation $\sO_X\otimes M$, introduced in Example \ref{Standard notation}, is justified in the following way. Since $\QCoh(X)$ is a stable stable $\i$-category, it is automatically tensored over  the $\i$-category of spectra $\Sp$. Consequently we can form tensoring with a spectrum $M$ for any object $\sF\in\QCoh(X)$ to obtain $\sF\otimes M\in\QCoh(X)$, determined completely by demanding that the functor $M\mapsto \sF\o M$ commutes with colimits and that $\sF\o S\simeq \sF$. Because quasi-coherent pullback $p^*:\Sp\to\QCoh(X)$ preserves colimits, and satisfies $f^*(S)\simeq \sO_X$ on account of being symmetric monoidal, it follows that it is indeed of the form $f^*(M)\simeq \sO_X\o M$ in terms of the tensoring over $\Sp$.
\end{remark}

 \begin{definition}\label{homotopy sheaves def}
 Let $X$ be a non-connective geometric stack, let $\sF\in\QCoh(X)$ be a quasi-coherent sheaf on it, and let $n\in\Z$ be an integer. The \textit{$n$-th homotopy sheaf of $\sF$} is a quasi-coherent sheaf $\pi_n\sF\in \QCoh(X^\heart)$, defined by the following property:
 \begin{itemize}[label = ($*$)]
 \item For some (and consequently any) choice of 
  a faithfully flat cover $ f:\Spec(A)\to X$, there is an equivalence
 $$
 (f^\heart)^*(\pi_n(\sF))\simeq \pi_n(f^*(\sF))
 $$
 in $\Mod_{\pi_0(A)}$.
 \end{itemize}
 \end{definition}

 \begin{remark}
If we choose a groupoid presentation
$X\simeq |\Spec (A^\bullet)|$ as in Proposition \ref{Groupoid presentation}, and the quasi-coherent sheaf $\sF\in\QCoh(X)$ corresponds under the equivalence of $\i$-categories $\QCoh(X)\simeq \Tot(\Mod_{A^\bullet})$ to the system of modules  $(M^\bullet\in \Mod_{A^\bullet})$, then the homotopy sheaf $\pi_n(\sF)\in\QCoh(X^\heart)$ corresponds to the system $(\pi_n(M^\bullet)\in \Mod^\heart_{\pi_0(A^\bullet)})$ under the equivalence of categories $\QCoh(X^\heart)\simeq \Tot(\Mod_{\pi_0(A^\bullet)})$.
\end{remark}

\begin{remark}
Note that, in light of Proposition \ref{Pres for trunc} and Definition \ref{homotopy sheaves def}, we have $\pi_0(\sO_X)\simeq \sO_{X^\heart}$ for any non-connective geometric stack $X$.
\end{remark}

Next we show that 
the pullback property of homotopy sheaves, used to define them above, holds more generally than just for a faithfully flat cover. That in particular implies the claimed independence of the definition from the choice of the cover.

\begin{lemma}\label{any flat pullback}
Let $f: X\to Y$ be an affine flat morphism of non-connective geometric stacks. For any $\sF\in \QCoh(Y)$ and any $t\in \mathbf Z$, there is a canonical isomorphism $\pi_t(f^*(\sF))\simeq (f^\heart)^*(\pi_t(\sF))$ in the $\i$-category $\QCoh(X^\heart)$.
\end{lemma}

\begin{proof}
Choose a flat affine cover $i:\Spec (A)\to Y$. Since $f$ is affine, $j : \Spec(A)\times_Y X\simeq \Spec(B)\to X$ is a flat affine cover too, and by the flatness of $f$, the $\E$-ring map $A\to B$ is flat.

By definition, the homotopy sheaf $\pi_t(\sF)\in \QCoh(X^\heart)$ is determined by satisfying the condition $(i^\heart)^*(\pi_t(\sF))\simeq \pi_t(i^*(\sF))$.
To show that $(f^\heart)^*(\pi_t(\sF))$ satisfies the analogous defining property of $\pi_t(f^*(\sF)),$ we consider the chain of equivalences
\begin{eqnarray*}
(j^\heart)^*(f^\heart)^*(\pi_t(\sF))&\simeq& \pi_0(B)\otimes_{\pi_0(A)} (i^\heart)^*(\pi_t(\sF))\\
&\simeq &  \pi_0(B)\otimes_{\pi_0(A)} \pi_t(i^*(\sF))\\
&\simeq & \pi_t(A\otimes_B i^*(\sF))\\
&\simeq & \pi_t(f^*(\sF)),
\end{eqnarray*}
of which the first equivalence uses commutativity of the  square
$$
\begin{tikzcd}
\Spec(\pi_0(B))\ar{d}{j^\heart} \ar{r} &\Spec(\pi_0(A))\ar{d}{i^\heart}\\
 X^\heart \ar{r}{f^\heart}& Y^\heart,
\end{tikzcd}
$$
(which is a pullback square thanks to the flatness hypothesis on $f$), the second equivalence is the already-discussed defining property of homotopy sheaves, the third equivalence is due to the flatness of the $\E$-ring map $A\to B$, and the final equivalence uses  an analogous commuting square to the one we displayed above, but removing $\heart$ and $\pi_0$.
\end{proof}

\begin{prop}\label{DSS}
Let $X$ be a non-connective geometric stack. For any
 quasi-coherent sheaf $\sF\in\QCoh(X)$, there exists an Adams-graded spectral sequence
$$
E^{s, t}_2 = \check{\mathrm H}^s(X^\heart; \pi_t(\sF))\Rightarrow \pi_{t-s}(\Gamma(X; \sF)),
$$
called the \textit{descent spectral sequence}. The second page is \v{C}ech cohomology for ordinary quasi-coherent sheaves on $X^\heart$.
\end{prop}

\begin{proof}
Choosing a presentation $X\simeq |\Spec(A^\bullet)|$ as in Proposition \ref{Groupoid presentation}, this is the Bousfield-Kan spectral sequence of the cosimplicial spectrum $\Gamma (\Spec(A^\bullet); \sF\vert_{\Spec (A^\bullet)})$. To identify the second and infinite page as in the statement,  repeat the proof of \cite[Lemma 3.1]{DevHop}.
\end{proof}

\subsection{Quasi-coherent sheaves in the connective setting}

The construction of the descent spectral sequence given above depends on the choice of a faithfully flat (hyper)cover. We wish to give an alternative construction of it that is manifestly independent of such a choice, but instead makes use of a $t$-structure (see Construction \ref{Cons of DSS}). Since the $\i$-category $\Mod_A$ of modules over a non-connective $\E$-ring does not carry a canonical $t$-structure, we can not expect one on $\QCoh(X)$ for a non-connective geometric stack either. There are no such issues in the connective setting however.

\begin{prop}
For any geometric stack $X$, this defines a right and left complete $t$-structure compatible with filtered colimits on the stable $\i$-category $\QCoh(X)$.
\end{prop}

\begin{proof}
Though we defined it  slightly differently, it follows from \cite[Remark 9.1.3.4]{SAG} that this $t$-structure on $\QCoh(X)$ coincides with the one studied in \cite[Subsection 9.1.3]{SAG}. Now the result we are after follows from \cite[Corollary 9.1.3.2]{SAG}.
\end{proof}

\begin{remark}
Picking a presentation $X\simeq |\Spec(A^\bullet)|$ as in Proposition \ref{Groupoid presentation}, the $t$-structure in question is explicitly given by $\QCoh(X)^{\le n}\simeq \Tot(\Mod_{A^\bullet}^{\le n})$. Note that this is sensible because all the degeneracy maps $A^n\to A^m$ are flat, and smash product along flat maps of connective $\E$-rings is left $t$-exact.
We thus see that the $t$-structure on quasi-coherent sheaves is induced via affines from the usual $t$-structure on module $\i$-categories over connective $\E$-rings of \cite[Proposition 7.1.1.13.]{HA}.
\end{remark}

 \begin{remark}
 Pullback along the map from the underlying ordinary stack $X^\heart\to X$ induces an equivalence on the heart of the $t$-structure $\QCoh(X)^\heart\simeq \QCoh(X^\heart)^\heart$. The latter may be thought as the ordinary abelian category of quasi-coherent sheaves on the underlying ordinary geometric stack $X^\heart$. In fact, even when $X$ is non-connective, it is still clear that $\pi_n(\sF)\in\QCoh(X^\heart)^\heart$ for every quasi-coherent sheaf $\sF\in\QCoh(X)$ and  $n\in \Z$.
 \end{remark}
 
 \begin{remark}
 Given any geometric stack $X\simeq |\Spec(A^\bullet)|$, presented as in Proposition \ref{Groupoid presentation}, we may define via descent
 $$
 \QCoh^\heart(|\Spec(A^\bullet|)\simeq \Tot(\Mod_{\pi_0(A^\bullet)}^\heart).
 $$
 This gives rise to an abelian category $\QCoh^\heart(X^\heart)$, equivalent to the heart $\QCoh(X^\heart)^\heart\simeq \QCoh(\tau_{\ge 0}(X))^\heart$ of the $t$-structure under discussion. In particular, we view $\QCoh^\heart(X^\heart)$ as the category of ordinary quasi-coherent sheaves on the ordinary geometric stack $X^\heart.$
 \end{remark}
 
To make use of the $t$-structure on quasi-coherent sheaves, we must pass from a non-connective geometric stack $X$ to its connective cover $\tau_{\ge 0}(X)$. That is to say, we consider pushforward $\QCoh(X)\to\QCoh(\tau_{\ge 0}(X))$ along the canonical map $X\to \tau_{\ge 0}(X)$. We will abuse notation and not notationally distinguish between a quasi-coherent sheaf on $X$ and its pushforward in $\tau_{\ge 0}(X)$. Indeed, said pushforward may be thought as a forgetful functor, as the following result shows.

\begin{prop}\label{QC and connective covers}
Let $X$ be a non-connective geometric stack. The map $X\to \tau_{\ge 0}(X)$ induces an equivalence of $\i$-categories
$$
\QCoh(X)\simeq \Mod_{\sO_X}(\QCoh(\tau_{\ge 0}(X))).
$$
\end{prop}

\begin{proof}
Denoting the connective cover map by $c: X\to \tau_{\ge 0}(X)$, we must show that the adjunction induced on quasi-coherent sheaves
$$
c^* :\QCoh(\tau_{\ge 0}(X))\rightleftarrows \QCoh(X) : c_*
$$
is monadic. Choosing a presentation $X\simeq |\Spec (A^\bullet)|$ as in Proposition \ref{Groupoid presentation}, we get by Corollary \ref{Pres for trunc} that $X\simeq |\Spec (\tau_{\ge 0}(A^\bullet))|$. The adjunction above is induced on totalizations from the adjunction of cosemisimplicial functors
$$
A^\bullet\otimes_{\tau_{\ge 0}(A^\bullet)} -:\Mod_{\tau_{\ge 0}(A^{\bullet})}\rightleftarrows \Mod_{A^\bullet},
$$
induced upon totalizations. In each degree separately, these adjunctions are  monadic, as special case of the basic fact that
 any map of $\E$-rings $A\to B$ induces an equivalence of $\i$-categories
$
\Mod_B(\Mod_A)\simeq \Mod_B.
$
Because all the face maps $A^i\to A^j$ of the simplicial $\E$-ring $A^\bullet$ are flat, the commutative diagram 
$$
\begin{tikzcd}
\tau_{\ge 0}(A^i) \ar[r]\ar[d] & \tau_{\ge 0}(A^j)\ar[d]\\
A^i\ar[r] & A^j
\end{tikzcd}
$$
is in fact a pushout square of $\E$-rings. This implies that the induced maps on the $\i$-categories of modules are adjointable in the sense of \cite[Definition 4.7.4.13]{HA} (said differently: satisfy the Beck-Chevalley condition), reducing the proof of monadicity to the following general Lemma.
\end{proof}

\begin{lemma}
Consider a small diagram $\mI\to \Fun(\Delta^1,\Pr^\mathrm{L})$, i.e.\ a collection of adjunction of presentable $\infty$-categories  $F_i : \mC_i\rightleftarrows \mD_i:G_i$ for all $i\in \mI$, and morphisms $f_{ij}:\mC_i\to\mC_j$ and $g_{ij} :\mD_i\to\mD_j$ for any morphism $i\to j$ in $\mI,$ such that all the diagrams of $\infty$-categories
$$
\begin{tikzcd}
\mC_i\ar{r}{F_i}\ar{d}{f_{ij}} & \mD_i\ar{d}{g_{ij}}\\
\mC_j \ar{r}{F_j} & \mD_j
\end{tikzcd}\qquad\qquad
\begin{tikzcd}
\mC_i\ar{d}{f_{ij}} & \mD_i\ar{d}{g_{ij}}\ar{l}[swap]{G_i}\\
\mC_j & \mD_j\ar{l}[swap]{G_j}
\end{tikzcd}
$$
commute (i.e. the commutative squares are right and left adjointable respectively.). Let us denote the limit of the functor $\mI\to \Fun(\Delta^1,\Pr^\mathrm{L})$ by $F:\mC\rightleftarrows \mD:G$. Suppose that each adjunction $F_i : \mC_i\rightleftarrows \mD_i:G_i$ is monadic for every $i\in\mI$. Then the adjunction $F: \mC\leftrightarrows \mD:G$ is also monadic.
\end{lemma}

\begin{proof}
Let us first observe that, since the forgetful functor $\Pr^\mathrm{L}\subseteq\Cat$, of presentable $\infty$-categories (with colimit-preserving functors) into all $\infty$-categories, preserves limits, we have canonical equivalences of $\i$-categories $\mC \simeq \varprojlim_{i\in \mI} \mC_i$ and $\mD \simeq\varprojlim_{i\in \mI} \mD_i$. Thanks to the adjointability of the adjunctions between $F_i$ and $G_i$, the functor $G:\mD\to\mC$ is induced from the functors $G_i:\mD_i\to\mC_i$ by \cite[Proposition 4.7.4.19]{HA}.

To prove that the adjunction  $F:\mC\rightleftarrows \mD:G$ is monadic, we must by the Barr-Beck Theorem \cite[Theorem 4.7.3.5]{HA} show that the functor $G:\mD\to\mC$ is conservative, and preserves $G$-split totalizations.
Recalling the definition of split simplicial objects from \cite[Definition 4.7.2.2]{HA}, the fact that $\Fun((\bDelta_{-\infty})^\mathrm{op}, \mC)\simeq \varprojlim_{i\in\mI}\Fun((\bDelta_{-\infty})^\mathrm{op}, \mC_i)$ implies that a $G$-split simplicial object $X^\bullet$ in $\mC$ corresponds to a functorial collection of $G_i$-split simplicial objects $X_i^\bullet$ in $\mC_i$. Since each $G_i$ preserves the totalization of $X_i^\bullet$, it follows that $G$ preserves the totalization of $X^\bullet$. To show that $G$ is conservative, consider the commutative diagram of $\infty$-categories
$$
\begin{tikzcd}
\mD\ar{r}{G}\ar{d}{\simeq} & \mC\ar{d}{\simeq}\\
\varprojlim_{i\in\mI}\mD_i\ar{d}\ar{r}{\varprojlim_{i\in\mI}G_i} &\varprojlim_{i\in\mI}\mC_i\ar{d}\\
\prod_{i\in\mI}\mD_i\ar{r}{\prod_{i\in\mI}G_i} &\prod_{i\in\mI}\mC_i.
\end{tikzcd}
$$
Here the unlabeled arrows are the usual inclusions of limits into products, which is to say, the functor induced on limits by the inclusion $\mathrm{ob}(\mI)\subseteq\mI$ of the set of objects into the indexing $\i$-category $\mI$. These functors are always conservative, and since each $G_i$ are conservative for all $i\in\mI$, so is $\prod_{i\in\mI}G_i$. It therefore follows from the commutativity the above diagram that $G$ must also be conservative.
\end{proof}

It follows in particular from (the proof of) Proposition \ref{QC and connective covers} that the homotopy sheaves $\pi_t(\sF)$ agree regardless of whether we start with $\sF\in\QCoh(X)$, or if we take its pushforward to $\QCoh(\tau_{\ge 0}(X))$. With that, we may give a construction of the descent spectral sequence that makes no reference to a choice of a flat (hyper)cover:

\begin{cons}\label{Cons of DSS}
Let $X$ be a non-connective geometric stack, and let $\sF\in\QCoh(X)$. Viewing $\sF$ as a quasi-coherent sheaf on the geometric stack $\tau_{\ge 0}(X)$, the Postnikov tower for the $t$-structure on $\QCoh(\tau_{\ge 0}(X))$ gives rise to the filtered object
$$
\mathbf Z\ni n\mapsto\Gamma(\tau_{\ge 0}(X);\tau_{\ge -n}(\sF))\in\QCoh(\tau_{\ge 0}(X)).
$$
Its associated spectral sequence is of the form (in the homological grading)
$$
E^{p,q}_1 = \pi_{p+q}(\Gamma(\tau_{\ge 0}(X); \Sigma^p(\pi_{-p}(\sF))))\Rightarrow \pi_{p+q}(\Gamma(X;\sF)).
$$
By re-grade via
$s= -(2p+q)$, $t =-p$, and $r\mapsto r+1,$ the spectral sequence is brought into (an Adams-graded) form
$$
E_2^{s,t} =\pi_{-s}(\Gamma(\tau_{\ge 0}(X); \pi_t(\sF)))\Rightarrow \pi_{t-s}(\Gamma(X; \sF)).
$$
Note that the homotopy sheaves $\pi_t(\sF)$ actually belong to $\QCoh(X^\heart)$, and are being secretly pushed forward along $X^\heart\to\tau_{\ge 0}(X)$. Thus $\Gamma(\tau_{\ge 0}(X); \pi_t(\sF))\simeq \Gamma(X^\heart; \pi_t(\sF))$, and since $\pi_t(\sF)\in \QCoh(X^\heart)^\heart$, this is just the complex computing sheaf cohomology of $\pi_t(\sF)$ on the ordinary stack $X^\heart$. With that, the spectral sequence becomes
$$
E^{s,t}_2 = \mathrm H^s(X^\heart; \pi_t(\sF))\Rightarrow \pi_{t-s}(\Gamma(X; \sF)).
$$
Under some light assumptions on $X^\heart$, which ensure that \v{C}ech cohomology agrees with derived-functor cohomology for quasi-coherent sheaves, the spectral sequence thus obtained is equivalent to the descent spectral sequence of Proposition \ref{DSS}. For a proof of such a claim in a related setting, see  \cite{Antieau}. 
\end{cons}

\section{Moduli of formal groups and chromatic homotopy theory}\label{Section 2}

Having taken the time to set up the necessary basics of non-connective spectral algebraic geometry, we now apply it to study a non-connective stack of particular interest to chromatic homotopy theory.

\subsection{Formal groups in spectral algebraic geometry}
Before getting to that though, we  review the theory of formal groups over $\E$-rings, as developed in \cite[Chapter 1]{Elliptic 2}, and needed in the rest of this paper. See \textit{loc.\,cit.}~for details and a more complete discussion.

\begin{definition}[{\cite[Definition 1.2.4]{Elliptic 2}}]\label{Smooth coalgebras}
A \textit{smooth coalgebra of dimension $r$} over an $\E$-ring $A$ is a cocommutative coalgebra object in the $\i$-category of flat $A$-modules
$$C\in\mathrm{cCAlg}_A^\flat\simeq \CAlg((\Mod_A^\flat)^\mathrm{op})^\mathrm{op},$$
such that  there exists a projective $\pi_0(A)$-module $E$ of finite rank $r$, and an isomorphism $\pi_0(C)\simeq \Gamma^*_{\pi_0(A)}(E)$ of coalgebras over $\pi_0(A)$. Here $\Gamma^*_{\pi_0(A)}(E)$ denotes the free divided power coalgebra, see for instance \cite[Construction 1.1.11]{Elliptic 2}. Smooth coalgebras over $A$ form a full subcategory $\mathrm{cCAlg}^\mathrm{sm}_A\subseteq\mathrm{cCAlg}_A^\flat$.
\end{definition}

The $\i$-category of smooth coalgebras $\mathrm{cCAlg}^\mathrm{sm}_A$ may be viewed as \textit{formal hyperplanes over $A$}, i.e.\ an incarnation of smooth formal varieties. Formal groups should therefore be defined to be some sort of commutative algebra objects in this $\i$-category. But we must be careful about picking the correct sort.

\begin{definition}[{\cite[Definition 1.2.4]{Elliptic 1}}]\label{Abelian group objects}
Let $\mathrm{Lat}$ denote the category of lattices, i.e.\ the full subcategory of abelian groups spanned by $\{\Z^n\}_{n\ge 0}$. Let $\mC$ be any $\i$-category with finite products. An \textit{abelian group object in $\mC$} is any functor $\mathrm{Lat}^\mathrm{op}\to\mC$ that preserves finite products. Abelian group objects form the full subcategory $\mathrm{Ab}(\mC)\subseteq \Fun(\mathrm{Lat}^\mathrm{op}, \mC)$.
\end{definition}

\begin{remark}
Contrast the notion of abelian group objects with the more familiar one of commutative (i.e.\ $\E$-)monoid objects in an $\i$-category with finite products $\mC$. The latter are given by  product-preserving functors $\mathcal F\mathrm{in}\to\mC$ from the category of finite sets. By pre-composing with the map $\mathcal F\mathrm{in}\to\mathrm{Lat}^\mathrm{op}$, given by $I\mapsto \mathbf Z^{I}$, we obtain a ``forgetful functor'' $\mathrm{Ab}(\mC)\to\mathrm{CMon}(\mC)$.
 For $\mC$ the $\i$-category of spaces $\mS$, this recovers the inclusion of topological abelian groups into $\E$-spaces.
\end{remark}

Though the definition of formal groups in \cite[Definition 1.6.1]{Elliptic 2} exhibits them manifestly as objects of functor-of-points-style algebraic geometry, it will be more convenient for us to use an equivalent ``Hopf algebra'' definition instead.

\begin{definition}[{\cite[Remark  1.6.6]{Elliptic 2}}]\label{formal group def}
The $\i$-category of \textit{formal groups} over an $\E$-ring $A$ is defined to be the $\i$-category $\mathrm{FGrp}(A) :=\mathrm{Ab}(\mathrm{cCAlg}^\mathrm{sm}_A)$ of abelian group objects in smooth coalgebras over $A$. The full subcategory of all formal groups whose underlying smooth coalgebras have dimension $r$ is denoted $\mathrm{FGrp}_{\dim =r}(A)$. 
\end{definition}

\begin{remark}
In \cite[Variant 1.6.2]{Elliptic 2}, formal groups over an $\E$-ring $A$ are \textit{defined} as formal groups over its connective cover $\tau_{\ge 0}(A)$. But since the extension of scalars $\mathrm{cCAlg}^\mathrm{sm}_A\to\mathrm{cCAlg}^\mathrm{sm}_{\tau_{\ge 0}(A)}$ is an equivalence of $\i$-categories by \cite[Proposition 1.2.8]{Elliptic 2}, it follows that our definition is no less general. Instead, we find that the canonical  functor $\mathrm{FGrp}(\tau_{\ge 0}(A))\to \mathrm{FGrp}(A)$ is an equivalence of $\i$-categories for any $\E$-ring $A$ as a \textit{consequence}, rather than  as definition.
\end{remark}

\begin{ex} Our interest in this paper is restricted to two classes of formal groups:
\begin{itemize}
\item Let $A$ be an ordinary commutative ring. If we view it as a discrete $\E$-ring, then formal groups over it, in the sense of Definition \ref{formal group def}, coincide with formal groups over $A$ in the usual sense, e.g.\ see \cite[Subsection 2.6]{Smithling}, \cite[Section 2]{Goerss} or \cite[Lecture 11]{Lurie Chromatic}.

\item Let $A$ be a \textit{complex periodic} $\E$-ring, i.e.\ complex orientable and such that $\pi_2(A)$ is a locally free $\pi_0(A)$ module of rank $1$. The \textit{Quillen formal group of $A$}, denoted $\w{\G}{}^{\CMcal Q}_A$, is in terms of Definition \ref{formal group def} defined to be the smooth coalgebra $C_*(\mathbf{CP}^\infty; A)$ over $A$. The abelian group object structure is inherited from the topological abelian group structure on $\mathbf{CP}^\infty$;
see \cite[Subsection 4.1.3]{Elliptic 2} for a proof that this actually defines a formal group over $A$.
\end{itemize}
\end{ex}

\begin{definition}[{\cite[Definition 4.3.9]{Elliptic 2}}]\label{second def of or}
Let $\w{\G}$ be a $1$-dimensional formal group over an $\E$-ring $A$. Let $\omega_{\w{\G}}$ be the dualizing line of $\w{\G}$, in the sense of \cite[Definition 4.2.14]{Elliptic 2}.
An \textit{orientation} on $\w{\G}$ is an equivalence $\omega_{\w{\G}}\simeq \Sigma^{-2}(A)$ in the $\i$-category $\Mod_A$, coming via the linearization construction of \cite[Construction 4.2.9]{Elliptic 2} from the choice of  an element in $\pi_2(\w{\G}(\tau_{\ge 0}(A))$ (which is also part of the orientation data). Let $\mathrm{FGrp}^\mathrm{or}(A)$ denote the $\i$-category of oriented formal groups over $A$.
\end{definition}

\begin{remark}\label{orientation motivation}
The notion of an oriented formal group may be motivated as follows. Suppose that $X$ is a non-connective spectral stack, classifying formal groups with perhaps some additional structure over $\E$-rings. Assume that $X$  satisfies conditions \ref{mtun} - \ref{mtquatre} of Theorem \ref{Maintheoremintrod} from the Introduction. 
 In light of the isomorphism \eqref{second page}, the condition \ref{mtquatre} amounts to demanding the equality of quasi-coherent sheaf cohomology groups
$$
\mathrm H^s(X^\heart; \pi_{2t}(\sO_X))\simeq \mathrm H^s(\mathcal M^\heart_\mathrm{FG}; \omega_{\mathcal M^\heart_\mathrm{FG}}^{\o t})
$$
for all $s\ge 0$, $t\in\mathbf Z$. Since we have by \ref{mtdeux} an equivalence of ordinary stacks $X^\heart\simeq \mathcal M^\heart_\mathrm{FG},$ this will be satisfied if there are isomorphisms of (usual) quasi-coherent sheaves 
$$\pi_{2t}(\sO_X)\simeq \omega_{\mathcal M^\heart_\mathrm{FG}}^{\o t}$$
for all $t\in \mathbf Z$. Let $A$ be an $\E$-ring and $\w{\G}$ be a formal group over $A$, which is classified by $X$. The quasi-coherent sheaf isomorphisms discussed above then give rise $\pi_0(A)$-module isomorphisms
$$
\pi_0(\Sigma^{-2t}(A))\simeq \pi_{2t}(A)\simeq\pi_0(\omega_{\w{\G}})^{\o t}\simeq \pi_0(\omega_{\w{\G}}^{\o t})
$$
for all $t\in \mathbf Z$. Fixing $t$, and looking at the left- and right-most terms, we might hope that this $\pi_0(A)$-module isomorphism was a reflection of an equivalence $\Sigma^{-2t}(A)\simeq \omega_{\w{\G}}^{\o t}$ on the level of $A$-module spectra. If we had such an equivalence for $t=1$, we could obtain it for all $t\in \mathbf Z$ by smash powers, consequently it suffices to assume that $\Sigma^{-2}(A)\simeq \omega_{\w{\G}}$. But of course, that is almost the notion of an orientation on the formal group $\w{\G}$, in the sense of Definition \ref{second def of or} (or equivalently, Definition \ref{Def of or} from the Introduction), ignoring only the technical issue of the isomorphism needing to arise from a homotopy class on ${\w{\G}}$.
\end{remark}

\begin{prop}[{\cite[Proposition 4.3.23]{Elliptic 2}}]\label{Paying proper credit}
Let $\w{\G}$ be a $1$-dimensional formal group over an $\E$-ring $A$. Then $\w{\G}$ is oriented if and only if $A$ is complex-periodic and $\w{\G}\simeq \w{\G}{}^{\CMcal Q}_A$.
\end{prop}

\begin{proof}[Proof sketch]
Consider  a \textit{preoriented} formal group $\w{\G}\in \mathrm{FGrp}(A)$, i.e.\ a formal group over $A$ equipped with a map $\beta: \omega_{\w{\G}}\to\Sigma^{-2}(A)$ in $\Mod_A$, arising via linearization as in \cite[Construction 4.2.9]{Elliptic 2} from an element in $\pi_2(\G(\tau_{\ge 0}(A)))$. Using the facts that that $\omega_{\w{\G}{}^{\CMcal Q}_A}\simeq C_\mathrm{red}^*(S^2; A)\simeq \Sigma^{-2}(A)$, and that $\mathbf{CP}^\infty$ is the free topological abelian group generated by the pointed space $\mathbf{CP}^1\simeq S^2$, the data of such a map $\beta$ may be seen to be equivalent to a map $\w{\G}{}^{\CMcal Q}_A\to\w{\G}$. Since a map of formal groups $\w{\G}\to\w{\G}{}'$ is an equivalence if and only if it induces an equivalence on dualizing lines $\omega_{\w{\G}{}'}\to\omega_{\w{\G}}$, it follows that an orientation on $\w{\G}$ is indeed equivalent to an equivalence of formal groups $\w{\G}{}^{\CMcal Q}_A\simeq \w{\G}$.
On the other hand, note that the Quillen formal group $\w{\G}{}^{\CMcal Q}_A$ is a $1$-dimensional formal group over $A$ if and only if  the $\E$-ring $A$ is complex-periodic.
\end{proof}

\begin{remark}\label{Why contractible}
For any formal group $\w{\G}$ over a complex-periodic $\E$-ring $A$, the space of  formal group maps $\w{\G}{}^{\CMcal Q}_A\to \w{\G}$ is by \cite[Proposition 4.3.21]{Elliptic 2} equivalent to the space of \textit{preorientations} on $\w{\G}$, i.e.\ maps $S^2\to\w{\G}(\tau_{\ge 0}(A))$ whose linearization $\beta:\omega_{\w{\G}}\to\Sigma^{-2}(A)$ is not  required to be an $A$-linear equivalence. Consequently, the Quillen formal group $\w{\G}{}^{\CMcal Q}_A$ of a complex-periodic $\E$-ring $A$ has no automorphisms \textit{as on oriented formal group over $A$}.
\end{remark}

\subsection{Descent for formal groups}

In this section we will show that various functors $\CAlg\to\Cat$ that we considered in the previous subsection satisfy faithfully flat descent.

\begin{prop}
\label{Descent for FG}
The functor $A\mapsto \mathrm{FGrp}(A)$ satisfies descent for the fpqc topology.
\end{prop}

\begin{proof}
The functor $A\mapsto\mathrm{Mod}_A$ satisfies flat descent by \cite[Theorem D.6.3.5]{SAG}, and since flatness if local for the fpqc topology, it follows that $A\mapsto \mathrm{Mod}_A^\flat$ does as well. The same thus holds for $A\mapsto \mathrm{cCAlg}_A^\flat\simeq \mathrm{cCAlg}(\Mod_A^\flat)$, and, because the construction of abelian group objects $\mathcal C\mapsto \mathrm{Ab}(\mathcal C)$ from Definition \ref{Abelian group objects} preserves limits, also for $A\mapsto \mathrm{Ab}(\mathrm{cCAlg}_A^\flat)$.
Now recall from Definition \ref{Smooth coalgebras} that the $\i$-category of smooth $\E$-coalgebras factors as
$$
\mathrm{cCAlg}_A^\mathrm{sm}\simeq \mathrm{cCAlg}_A^\flat\times_{\mathrm{cCAlg}^\flat_{\pi_0(A)}}\mathrm{cCAlg}_{\pi_0(A)}^\mathrm{sm}.
$$
Thanks to the already-mentioned fact that $\mathcal C\mapsto\mathrm{Ab}(\mathcal C)$ commutes with limits,
and we may write $\mathrm{FGrp}(A)\simeq \mathrm{Ab}(\mathrm{cCAlg}^\mathrm{sm}_A)$ by  \cite[Remark 1.6.6]{Elliptic 2}, we find that
$$
\mathrm{FGrp}(A)\simeq \mathrm{Ab}(\mathrm{cCAlg}_A^\flat)\times_{\mathrm{Ab}(\mathrm{cCAlg}^\flat_{\pi_0(A)})}\mathrm{Ab}\big(\mathrm{cCAlg}_{\pi_0(A)}^\mathrm{sm}\big).
$$
It now suffices to show that $A\mapsto \mathrm{Ab}\big(\mathrm{cCAlg}_{\pi_0(A)}^\mathrm{sm}\big)$ satisfies fpqc descent. By design, the construction $ \mathrm{Ab}\big(\mathrm{cCAlg}_{R}^\mathrm{sm}\big)\simeq \mathrm{FGrp}(R)$ recovers the usual category of formal groups for an ordinary commutative ring $R$. This can be outsourced e.g.\ to \cite[Corollary 2.6.6]{Smithling}, or \cite[Theorem 2.30]{Goerss} (though the assumption there is that we are only considering $1$-dimensional formal groups, the argument works just as well for arbitrary finite-dimensional ones), or easily verified directly by using \cite[Lemma 1.1.20]{Elliptic 2} - applicable since any Hopf algebra has at least one grouplike element: its multiplicative unit.
\end{proof}

\begin{remark}
The functor $A\mapsto\mathrm{cCAlg}_A^\mathrm{sm}$  only satisfies \'etale descent, but not fpqc descent. The issue is addressed in
\cite[Warning 1.1.22]{Elliptic 2}, and boils down to the possibility of a coalgebra having no grouplike elements. As we observed in the proof above however, formal groups correspond to Hopf algebras, for which this issue does not arise. 
\end{remark}

\begin{corollary}\label{Descent for FG1}
The functor $A\mapsto\mathrm{FGrp}_{\dim =r}(A)$ satisfies descent for the fpqc topology for any $r\ge 0$.
\end{corollary}

\begin{proof}
Follows from from Proposition \ref{Descent for FG} and fpqc descent for finitely generated projective modules over ordinary commutative rings \cite[Proposition 2.5.2]{EGA4} or \cite[\href{https://stacks.math.columbia.edu/tag/10.83}{Tag 10.83}]{stacks-project}.
\end{proof}

\begin{prop}\label{Descent for FGor}
The functor $A\mapsto \mathrm{FGrp}^\mathrm{or}(A)$ satisfies descent for the fpqc topology.
\end{prop}

\begin{proof}
Consider the forgetful functor $\mathrm{FGrp}^\mathrm{or}\to\mathrm{FGrp}_{\dim =1}$. 
A map $\Spec (A)\to \mathrm{FGrp}_{\dim =1}$ is equivalent to specifying a $1$-dimensional formal group $\w{\G}\in \mathrm{FGrp}(A)$, and the pullback $\Spec(A) \times_{\mathrm{FGrp}_\mathrm{\dim =1}}\mathrm{FGrp}^\mathrm{or}$, viewed as a functor $\CAlg_A\to\mS$,
is by definition given by $B\mapsto\mathrm{OrDat}(\w{\G}_B)$. By \cite[Proposition  4.3.13]{Elliptic 2} we have $\mathrm{OrDat}(\w{\G}_B)\simeq \Map_{\CAlg_A}(\mathfrak O_{\w{\G}}, B)$ for an $\mathbb E_\infty$-algebra $\mathfrak O_{\w{\G}}$ over $A$. Thus $\Spec(A) \times_{\mathrm{FGrp}_{\dim = 1}}\mathrm{FGrp}^\mathrm{or}\simeq \Spec(\mathfrak O_{\w{\G}})$ is (representable by) an affine spectral $A$-scheme, and so the map $\mathrm{FGrp}^\mathrm{or}\to\mathrm{FGrp}_{\dim =1}$ is affine. Thus  $\mathrm{FGrp}^\mathrm{or}$ satisfies fpqc descent by virtue of that holding for $\mathrm{FGrp}$ by Corollary \ref{Descent for FG1}.
\end{proof}

\begin{remark}\label{MorFG explicitly}
Since $\mathrm{FGrp}^\mathrm{or}(A)$ is either contractible if the $\E$-ring $A$ is complex-periodic (see Remark \ref{Why contractible}) or empty otherwise, Proposition \ref{Descent for FGor} is equivalent to the following assertion: if an $\E$-ring $A$ admits a faithfully flat $\E$-ring map $A\to B$ into a complex-periodic $\E$-ring $B$, then $A$ is complex-periodic. The analogous statement for weak $2$-periodicity is tautological from the definition of flatness for $\E$-rings, but complex orientability is less immediate.
\end{remark}

In light of the preceding results, we will refer to the functors $A\mapsto \mathrm{FGrp}_{\dim=1}(A)^\simeq$ and  $A\mapsto\mathrm{FGrp}^\mathrm{or}(A)\simeq \mathrm{FGrp}^\mathrm{or}(A)^\simeq$ as  the \textit{moduli stack of (resp.\ oriented) formal groups}, and denote them by $\mathcal M_\mathrm{FG},\mathcal M^\mathrm{or}_\mathrm{FG}\in \mathcal S\mathrm{hv}_{\mathrm{fpqc}}^\mathrm{nc}$.

\begin{remark}
With this terminology, we embrace the tradition in homotopy theory to drop the adjective ``$1$-dimensional'' from the term ``$1$-dimensional formal group'', since all the formal groups we will care about will be $1$-dimensional.
\end{remark}

\subsection{The stack of oriented formal groups}
This section is dedicated to proving the following result:

\begin{theorem}\label{M is geometric}
The moduli stack of oriented formal groups $\M$ is a non-connective geometric stack.
\end{theorem}

We already showed that $\M$ satisfies fpqc descent in Proposition \ref{Descent for FGor}. According to Definition \ref{Def of geom stack}, it remains to prove that its has affine diagonal, and that it admits a faithfully flat cover by an affine. The first of those is easy:

\begin{lemma}\label{affine diagonal lemma}
The diagonal morphism $\M\to \M\times \M$ is affine.
\end{lemma}

\begin{proof}
For any $\E$-ring $A$, there exists an essentially unique map $\Spec(\mathrm{A})\to \mathcal M^\mathrm{or}_\mathrm{FG}$ if and only if $A$ is complex-periodic. It follows from this that, for any pair of complex-periodic $\E$-rings $A$ and $B$, the fibered product $\Spec(A)\times_{\mathcal M_\mathrm{FG}^\mathrm{or}}\Spec(B)$ is equivalent to
$\Spec(A\otimes B)$. By Remark \ref{Why affine diagonal?}, this is what we needed to show.
\end{proof}

For the purpose of finding an affine atlas for $\M$, we introduce a distinguished class of $\E$-rings. Recall e.g.\ from \cite[Remark 2.2]{Hahn-Yuan} that the $2$-periodic spectrum $ S[\beta^{\pm 1}] :=\bigoplus_{n\in \mathbf Z}\Sigma^{2n}(S)$ carries a canonical $\mathbb E_2$-ring structure by \cite[Remark 2.2]{Hahn-Yuan}.

\begin{definition}
An $\E$-ring is  a \textit{form of periodic complex bordism} if it is equivalent to $\mathrm{MU}\otimes S[\beta^{\pm 1}]$ as an $\mathbb E_2$-ring.
\end{definition}

\begin{ex}\label{The MPs}
The following constructions all give rise to forms of periodic complex bordism, and by \cite[Theorem 1.3]{Hahn-Yuan}, at least the first two are not equivalent to each other as $\E$-rings.
\begin{itemize}
\item The Thom spectrum $$\mathrm{MUP}\simeq \varinjlim \big(\mathrm{BU}\times \mathbf Z\simeq (\mathrm{Vect}_{\mathbf C}^\simeq)^\mathrm{gp}\xrightarrow{J} \mathrm{Pic}(\Sp)\subseteq \Sp\big),$$
given by the colimit of the map induced upon group completion from the symmetric monoidal functor $\mathrm{Vect}^\simeq_{\mathbf C}\to \mathrm{Pic}(\Sp)$ given by $V\mapsto S^V$. This is equivalent to the $\E$-ring appearing under the name $\mathrm{MP}$ in \cite{Elliptic 2}.
\item Snaith's constuction $(S[\mathrm{BU}])[\beta^{-1}]$, obtained by inverting the Bott element $\beta \in \pi_2(S[\mathrm{BU}])$, induced from $S^2\simeq \mathbf{CP}^1\subseteq \mathbf{CP}^\infty\simeq \mathrm{BU}(1)\to \mathrm{BU}$, in the suspension spectrum $S[\mathrm{BU}]$, which inherits the structure of an $\E$-ring from the infinite loop space structure on $\mathrm{BU}\simeq \Omega^\infty (\Sigma^2(\mathrm{ku}))$.

\item The Tate spectrum $\mathrm{MU}^{tS^1}$, with its $\E$-ring structure induced via the lax symmetric monoidality of the Tate construction from the usual Thom spectrum $\mathbb E_\infty$-ring structure on the  complex bordism spectrum $\mathrm{MU}\simeq \varinjlim(\mathrm{BU}\xrightarrow{J}\Sp)$. According to \cite[Remark 1.13]{Hahn-Yuan}, this form of periodic complex bordism was suggested by Tyler Lawson.
\end{itemize}
\end{ex}

From now on, and throughout the rest of this note, we fix $\mathrm{MP}$ to be an arbitrary form of periodic complex bordism.

\begin{lemma}\label{Quillen}
There exists an essentially unique map $\Spec(\mathrm{MP})\to\mathcal{M}^\mathrm{or}_\mathrm{FG}$, and this map is  faithfully flat.
\end{lemma}

\begin{proof}
The existence of a map $\Spec(\MP)\to\M$ is equivalent to $\mathrm{MP}$ being complex-periodic, a condition which only depends on the underlying commutative ring spectrum. The complex bordism spectrum $\mathrm{MU}$ is the universal complex orientable spectrum, 
 and the spectrum $S[\beta^{\pm 1}]$ is $2$-periodic, implying that both of those hold for the smash product $\mathrm{MU}\otimes S[\beta^{\pm 1}]\simeq \MP$ as well.
 
 To show that the map in question is faithfully flat, it suffices to prove that after pullback along an arbitrary map $\Spec(A)\to \M$.  As we saw in the proof of Lemma \ref{affine diagonal lemma}, this means proving that assertion the map $A\to A\otimes\MP$ is faithfully flat for every complex-periodic $\E$-ring $A$. That is the content of \cite[Theorem 5.3.13]{Elliptic 2}, but it is really just the result of a standard computation in chromatic homotopy theory; that of the $\mathrm{MU}$-homology of a complex oriented ring spectrum  \cite[Lemma 4.1.7, Lemma 4.1.8, and Corollary 4.1.9]{Green Book}, \cite[Lecture 7]{Lurie Chromatic}, or \cite[Proposition 6.2]{COCTALOS},
\end{proof}

\begin{remark}
In \cite[Theorem 5.3.13]{Elliptic 2}, the claim we used in the proof above is only asserted for a specific choice of a periodic form of complex bordism $\mathrm{MP}\simeq\mathrm{MUP}$. But since being faithfully flat is a condition fully determined on the level of the underlying homotopy-commutative ring spectra, that restriction makes no difference.
\end{remark}

With that, Theorem \ref{M is geometric} is proven, and we may begin to reap the rewards.

\begin{corollary}\label{Main Thm}
There is a canonical equivalence
$$
\mathcal M^\mathrm{or}_\mathrm{FG}\simeq |\mathrm{\check C}^\bullet(\Spec(\mathrm{MP})/\Spec(S))|,
$$
with the geometric realization formed in  the $\i$-topos $\mathcal S\mathrm{hv}_{\mathrm{fpqc}}^\mathrm{nc}$.
\end{corollary}

\begin{proof}
This is an instance of Proposition \ref{Groupoid presentation}, or more specifically, an application of Lemma \ref{covtriv}.
\end{proof}

\begin{remark}
The \v{C}ech groupoid appearing in the statement of Corollary \ref{Main Thm} is the so-called Amitsur complex. It encodes the descent data along the terminal map of stacks $\Spec(\MP)\to\Spec(S)$. But since this map is not flat itself (i.e.\ $\MP$ is not a flat spectrum), descent along it does not return $\Spec(S)$ itself, but rather $\M$ as we saw.
\end{remark}

\begin{corollary}\label{Underlying ordinary stack is MFG}
The canonical map $\M\to\mathcal M_\mathrm{FG}$ induces an equivalence on underlying ordinary stacks. In particular, $(\M)^\heart\simeq \mathcal M_\mathrm{FG}^\heart$ is the ordinary moduli stack of formal groups.
\end{corollary}

\begin{proof}
It follows from Corollary \ref{Main Thm} and Corollary \ref{Pres for trunc} that the underlying ordinary stack of $\M$ is given by
$$
(\M)^\heart\simeq \left|\Spec \big(\pi_0\big(\mathrm{MP}^{\otimes (\bullet +1)}\big ) \big)\right|.
$$
By Quillen's Theorem, we have canonical isomorphisms $\pi_0(\mathrm{MP})\simeq L$ with  the Lazard ring, classifying formal group laws, and $\pi_0(\mathrm{MP}\otimes \mathrm{MP})\simeq W$ is the groupoid of (non-strict) isomorphisms of formal group laws. In particular, the standard groupoid presentation of the ordinary moduli stack of formal groups
$$
\mathcal M^\heart_\mathrm{FG}\simeq \varinjlim(\Spec(W)\rightrightarrows \Spec (L)),
$$
e.g.\ from \cite[Theorem 2.6.4]{Smithling} or \cite[Theorem 2.3.4]{Goerss}, gives rise to the desired identification $(\M)^\heart\simeq \mathcal M^\heart_\mathrm{FG}$ - see also \cite[Example  9.3.1.8]{SAG}.
\end{proof}

\begin{remark}\label{Heart is fickle}
The underlying ordinary stack $(\M)^\heart\simeq \mathcal M_\mathrm{FG}^\heart$ is the classical stack of formal group. As a functor $\mathcal M_\mathrm{FG}^\heart:\CAlg^\heart\to\mathcal G\mathrm{rpd}\simeq\tau_{\le 1}(\mS)\subseteq \mS$,  it is highly non-trivial. It is thus very far from the restriction we have $\M\vert_{\CAlg^\heart}$ onto the subcategory $\CAlg^\heart\subseteq \CAlg$ of discrete $\E$-rings, since $\M(A)=\emptyset$ for any $\E$-ring that is not complex-periodic, which includes every ordinary commutative ring.
\end{remark}

\subsection{Quasi-coherent sheaves on the  stack of oriented formal groups}

Let us begin this section with a straightforward computation, going back ostensibly to Bousfield \cite[Theorem 6.5]{Bousfield}, implied by the results of the previous one.
\begin{prop}\label{functions are S}
The initial map of $\E$-rings $S\to\mathcal O(\mathcal M_{\mathrm{FG}}^\mathrm{or})$ is an equivalence.
\end{prop}

\begin{proof} By applying the limit-preserving functor $\sO :(\Shv_\mathrm{fpqc}^\mathrm{nc})^\mathrm{op}\to\CAlg$ to the equivalence of Corollary \ref{Main Thm}, we obtain an equivalence of $\E$-rings
\begin{eqnarray*}
\sO(\mathcal M_\mathrm{FG}^\mathrm{or}) &\simeq&
\sO\big(|\check{\mathrm C}^\bullet(\Spec (\mathrm{MP})/\Spec(S))|\big)\\
&\simeq&
\Tot\big(\sO\big(\Spec (\mathrm {MP})^{\times (\bullet+1)}\big)\big)\\
&\simeq & \Tot\big(\mathrm{MP}^{\otimes(\bullet+1)}\big) =:  S^\wedge_{\mathrm{MP}}
\end{eqnarray*}
with the nilpotent completion of the sphere spectrum along $\mathrm{MP}$. We thus need to show that the canonical map of $\E$-rings $S\to S^\wedge_\mathrm{MP}$ is an equivalence.
Nilpotent completion depends only on the $\mathbb E_2$-structure on $\mathrm{MP}$, hence it suffices to choose any particular form of periodic complex bordism.
For the Thom spectrum $\mathrm{MUP}$, introduced in Example \ref{The MPs}, the unit map $S\to \mathrm{MUP}$ is a Hopf-Galois extension by \cite[Remark 12.2.3]{Rognes}, from which the desired completion claim follows from \cite[Proposition 12.1.8]{Rognes}.
\end{proof}

This  suggests a close connection between quasi-coherent sheaves on $\M$ and the $\i$-category of spectra. To make that precise with Theorem \ref{IndCoh is Sp}, we need to first single out a certain class of quasi-coherent sheaves.

\begin{definition}\label{Def of Coh}
For any fpqc stack $X$, let $\IndCoh(X)$ denote the ind-completion of  the thick subcategory of $\QCoh(X)$, spanned by the structure sheaf $\sO_X$, i.e.\ the smallest stable full subcategory of $\QCoh(X)$ that is contains $\sO_X$ and is closed under retracts.
\end{definition}

\begin{remark}
The preceding definition, partially inspired by \cite[Definition 5.39]{Bartel Heard Valenzuela} is quite fanciful. In particular, it will almost certainly fail to specify many of the good properties of its namesake from derived algebraic geometry, as discussed for instance in \cite{GaRo}. But in our context, with our very specific scope of interest, the above terminology is highly convenient and at least somewhat appropriate to indicate the difference of the $\i$-category in question from $\QCoh(X)$.
\end{remark}

\begin{theorem}\label{IndCoh is Sp} The functor $\Sp\to\QCoh(\M)$ given by  $X\mapsto \sO_{\M}\otimes X$ induces an equivalence of  $\i$-categories
$\IndCoh(\mathcal M^\mathrm{or}_\mathrm{FG})\simeq \Sp$.
\end{theorem}

\begin{proof}
 Recall from Example \ref{Standard notation} that quasi-coherent sheaves on $\M$ and spectra are  related by the adjunction
\begin{equation}\label{adjunction on QCoh}
\sO_{\M}\otimes - : \Sp \rightleftarrows \QCoh(\M) : \Gamma(\M; -),
\end{equation}
obtained by quasi-coherent pullback and pushforward along the terminal map of stacks $\M\to\Spec(S)$. We can identify the unit of this adjunction on the subcategory $\Sp^\mathrm{fin}\subseteq\Sp$ as the identity functor by Lemma \ref{also for compacts}. Therefore the restriction of the left adjoint to this subcategory is fully faithful. On the other hand, its essential image is precisely the thick subcategory in $\QCoh(X)$ spanned by $\sO_X$. In light of Definition \ref{Def of Coh}, the result follows by passing to the ind-completion on both sides.
\end{proof}

\begin{lemma}\label{also for compacts}
For any finite spectrum $M$, the canonical map $M\to \Gamma(\M; \sO_{\M}\otimes M)$ is an equivalence of spectra.
\end{lemma}

\begin{proof}
The functors $\sF\mapsto \Gamma(\M; \sF)$ and $M\mapsto \sO_{\M}\otimes M$ form an adjunction, and therefore preserve colimits and limits respectively.
Since both $\Sp$ and $\QCoh(\M)$ are stable $\i$-categories, this implies that both functors preserve finite (co)limits. Recalling that the subcategory of finite spectra $\Sp^\mathrm{fin}\subseteq \Sp$ is spanned by the sphere spectrum $S$ under finite colimits, we have reduced to proving the claim for $M=S$, in which case it follows from Proposition \ref{functions are S}.
\end{proof}

\begin{remark}
The conclusion of Lemma \ref{also for compacts} may, in light of the proof of Proposition \ref{functions are S}, 
be restated as the claim that the canonical map $M\to M^\wedge_{\mathrm{MP}}$ is an equivalence of spectra for any $M\in\Sp^\mathrm{fin}$. That is to say, every finite spectrum is $\mathrm{MP}$-nilpotent complete. 
\end{remark}

\begin{remark}
Another way to view Theorem \ref{IndCoh is Sp} is as an instance of the  Schwede-Shipley Recognition Theorem \cite[Theorem  7.1.2.1]{HA} for compactly generated stable $\i$-categories. Indeed, the $\i$-category  $\IndCoh(\M)$ is stable and compactly generated by the structure sheaf $\sO_{\M}$ by construction, hence the aforementioned result identifies it with $\Mod_{\sO(\M)}$, which Lemma \ref{functions are S} shows to be $\Sp$.
\end{remark}

We should also mention the comodule-theoretic description of the $\i$-category of quasi-coherent sheaves on $\M$, analogous to the classical comodule-theoretic description of the abelian category of ordinary quasi-coherent sheaves on $\mathcal M^\heart_\mathrm{FG}$ that underlies most  approaches to the ANSS.

\begin{cons}\label{comodule cat const}
Let $A$ be an $\E$-ring. 
The forgetful functor $\Mod_A\to\Sp$ admits a left adjoint $M\mapsto M\o A$. This adjunction induces the comonad $T : \Mod_{A}\to\Mod_{A}$ given by
$$
T(M)\simeq M\o A\simeq M\o_A(A\o A).
$$
The comonad structure on $T$ is therefore equivalent to an $\mathbb E_1$-coalgebra structure on $A\o A$ over $A$ (see \cite{Torii} for a precise treatment thereof). The \textit{$\i$-category of comodules} over it, defined as
$$
\cMod_{A \o A}(\Mod_A):= \Mod_T(\Mod_{A}^\mathrm{op})^\mathrm{op},
$$
may be identified by the Beck-Chevalley theory of comonadic descent, in the form of \cite[Theorem 4.7.5.2]{HA}, with the totalization $\Tot(A^{\otimes(\bullet+1)})$ of the Amitsur complex of $A$. Indeed, the latter may be identified with the cobar construction of the coalgebra $A\o A$ over $A$.
\end{cons}

\begin{prop}\label{QCoh as comod}
Pullback along the cover $\Spec(\MP)\to\M$ induces an equivalence of $\i$-categories $\QCoh(\M)\simeq \cMod_{\MP\otimes\MP}(\Mod_{\MP}).$ By passing to $\pi_0$, this recovers the traditional equivalence of abelian categories $\QCoh^\heart(\mathcal M^\heart_\mathrm{FG})\simeq \cMod_{W}(\Mod_L^\heart)$.
\end{prop}

\begin{proof}
Follows from the Beck-Chevalley description of the comodule $\i$-category, since we already know by
Corollary \ref{Main Thm} that $\QCoh(\M)\simeq \Tot\big(\Mod_{\MP^{\otimes (\bullet +1)}}\big)$.
\end{proof}

\begin{remark}
In terms of the comodule description of quasi-coherent sheaves on $\M$, the $\i$-category $\IndCoh(\M)$ is a version of Hovey's stable category of comodules; see \cite{Bartel Heard Valenzuela} for a discussion. A modern treatment can be found in \cite[Definition 2.4]{Krause}, where this is called the \textit{compactly generated category of comodules}. Our Theorem \ref{IndCoh is Sp} may be viewed as a special case of Krause's \cite[Theorem 2.44]{Krause}; see in particular \cite[Example 2.45]{Krause} for a version with $\MU$ instead of $\MP$.
\end{remark}

\subsection{Landweber exactness}\label{LEFT}
Let $X$ be any fixed spectrum. The non-connective geometric stack $\M$ allows us to define a quasi-coherent sheaf
$$\sF_n(X) := \pi_n(\sO_{\M}\otimes X)\in\QCoh^\heart(\mathcal M^\heart_\mathrm{FG}).$$
It satisfies
$\sF(\Sigma^{-n} (X))\simeq \pi_n(\sO_{\M}\otimes X),$ hence it is encoding the $\QCoh^\heart(\mathcal M^\heart_\mathrm{FG})$-valued homology theory corepresented by the structure sheaf $\sO_{\M}.$

\begin{remark}
The collection of sheaves $\sF_n(X)$ for all spectra $X$ contain already for $n=0$ all the information encoded in those for other $n\in\Z$, since
$$
\sF_n(X)\simeq \pi_0(\Sigma^{-n}(\sO_{\M}\o X))\simeq \pi_0(\sO_{\M}\o \Sigma^{-n}(X))\simeq \sF_0(\Sigma^{-n}(X)).
$$
\end{remark}

\begin{remark}
Let $\omega_{\M}\in \QCoh(\M)$ be the dualizing line, i.e.\ the module of invariant differentials on the universal oriented formal group, pulled back from the analogous $\omega_{\mathcal M_\mathrm{FG}}$ along the canonical map $\M\to\mathcal M_\mathrm{FG}$. These quasi-coherent sheaves are flat by definition, and  in follows from Corollary \ref{Underlying ordinary stack is MFG} that $\pi_0(\omega_{\M})\simeq \omega_{\mathcal M_\mathrm{FG}^\heart}\in\QCoh^\heart(\mathcal M^\heart_\mathrm{FG})$ is the sheaf of invariant differentials on the universal ordinary formal group. On the other hand, we have by the definition of orientability for formal groups that $\omega_{\M}\simeq\Sigma^{-2}(\sO_{\M})$, and consequently $\omega_{\M}^{\o n}\simeq \Sigma^{-2n}(\sO_{\M})$ for every $n\in\Z$. For any spectrum $X$ and any $n\in \Z$, it follows that
\begin{eqnarray*}
\pi_{2n}(X)&\simeq& \pi_0(\Sigma^{-2n}(\sO_{\M})\o X) \\
&\simeq& \pi_0(\omega_{\M}^{\otimes n}\o X)\\
&\simeq &\omega^{\o n}_{\mathcal M^\heart_\mathrm{FG}}\otimes_{\sO_{\mathcal M^\heart_\mathrm{FG}}}\pi_0(\sO_{\M}\o X)\\
&\simeq &\omega^{\o n}_{\mathcal M^\heart_\mathrm{FG}}\otimes_{\sO_{\mathcal M^\heart_\mathrm{FG}}}\sF_0(X),
\end{eqnarray*}
and, either by an analogous consideration or by use of the last remark, similarly for odd-degree homotopy groups
$$
\pi_{2n+1}(X)\simeq \omega^{\o n}_{\mathcal M^\heart_\mathrm{FG}}\otimes_{\sO_{\mathcal M^\heart_\mathrm{FG}}}\sF_1(X).
$$
Thus already for any fixed spectrum $X$, the sheaves $\sF_0(X)$ and $\sF_1(X)$ contain the information of all the other $\sF_n(X)$.
\end{remark}

\begin{prop}\label{sheaves F}
Let $f^\heart:\Spec(L)\to\mathcal M_\mathrm{FG}^\heart$ be the cover of the ordinary stack of formal groups by formal group laws. For any spectrum $X$, there is a canonical isomorphism of $L$-modules $(f^\heart)^*(\sF_n(X))\simeq \MP_n(X)$.
\end{prop}

\begin{proof}
The cover in question is induced upon the underlying ordinary stacks by the cover of non-connective geometric stacks $f:\Spec(\MP)\to \M$. Using Lemma  \ref{any flat pullback}, and the fact that quasi-coherent pullback commutes with colimits, we obtain the series of natural equivalences
\begin{eqnarray*}
(f^\heart)^*(\sF_n(X))&\simeq & (f^\heart)^*\pi_n(\sO_{\M}\o X)\\
&\simeq& \pi_n(f^*(\sO_{\M}\otimes X))\\
&\simeq &\pi_n(f^*(\sO_{\M})\otimes X)\\
&\simeq& \pi_n(\MP\otimes X),
\end{eqnarray*}
where the final term is   the $n$-th $\MP$-homology of the spectrum $X$ by definition.
\end{proof}

\begin{remark}
It follows from Proposition \ref{sheaves F} that the sheaves $\sF_n(X)$ on $\mathcal M^\heart_\mathrm{FG}$ may equivalently be defined in terms of the usual coordinatized presentation of formal groups in terms of formal group laws, which is to say the groupoid presentation of the moduli stack $\mathcal M^\heart_\mathrm{FG}\simeq\colim (\Spec(W)\rightrightarrows \Spec(L)).$
In combination with the isomorphism of Hopf algebroids $(\pi_0(\MP), \pi_0(\MP\o \MP))\simeq (L, W)$ of Quillen's Theorem, the sheaf $\sF_n(X)$ therefore corresponds to the $L = \pi_0(\MP)$-module $\MP_n(X) = \pi_0(\MP\otimes X)$, equipped with its usual  of $W = \MP_0(\MP)\simeq \pi_0(\MP\otimes \MP)$ as a generalized Steenrod algebra. It is in this guise that the sheaves $\sF_n(X)$ appear as foundations for chromatic homotopy theory in \cite{Lurie Chromatic}. The advantage of our approach is that it does not require explicitly distinguishing the complex bordism spectrum $\MP$, but instead proceeded from the non-connective spectral stack $\M$.
\end{remark}

\begin{remark}\label{MP as a homology theory}
From another perspective, Proposition \ref{sheaves F} allows us to construct the homology theory $\MP_*$ corresponding to the periodic complex bordism spectrum solely in terms of the stack $\M$ and the classical stack $\mathcal M_\mathrm{FG}^{\heart, \mathrm{coord}}$ of coordinatized classical formal groups (equivalently: formal group laws - see Definition \ref{coordinatized def})
$$
\MP_n(X)\simeq (f^\heart)^*(\pi_n(\sO_{\M}\o X))\simeq \Gamma\big(\mathcal M^\heart_\mathrm{FG}; \pi_n(\sO_{\M}\otimes X)\otimes_{\sO_{\mathcal M^\heart_\mathrm{FG}}}
\sO_{\mathcal M_\mathrm{FG}^{\heart, \mathrm{coord}}}\big).
$$
This might not recover the $\E$-ring structure, but then we should not expect it would; after all, we are working with an arbitrary form of periodic complex bordism $\MP$.
\end{remark}

\begin{remark}
The construction of the preceding Remark hints at an attractive possibility of a path towards an alternative, perhaps more insightful, proof of Quillen's Theorem. Nevertheless, it relies crucially on the observation that the underlying ordinary stack of $\M$ is the ordinary stack of formal groups $\mathcal M^\heart_\mathrm{FG}$, our proof of which in Corollary \ref{Underlying ordinary stack is MFG} ultimately reduced to an application of Quillen's Theorem. If we were able to leverage the abstract spectral-algebro-geometric description of $\M$ to identify its underlying stack with  $\mathcal M^\heart_\mathrm{FG}$ though, we believe that Quillen's Theorem would follow. In that sense, while we have not succeeded to provide a more insightful proof of Quillen's Theorem, we have succeeded to reduce it to a purely spectral-algebro-geometric statement (that of Corollary \ref{Underlying ordinary stack is MFG}) which does not involve complex bordisms.
\end{remark}

The example of Remark \ref{MP as a homology theory} may be extended to show that all Landweber exact homology theories come directly from $\M$.

\begin{corollary}
Let $R$ be a commutative ring and $\w{\G}\in \mathrm{FGrp}_{\dim =1}(R)$ a Landweber-exact $1$-dimensional formal group law. That is to say, suppose that  its classifying map $\eta_{\w{\G}}:\Spec(R)\to\mathcal M_\mathrm{FG}^\heart$ is flat. If the associated even-periodic Landweber exact spectrum is denoted $E_R$, then its value as a homology theory $(E_R)_n :\Sp\to\mathrm{Ab}$ is given by
$$
(E_R)_n(X)\simeq \eta_{\w{\G}}^*(\sF_n(X))
$$
\end{corollary}

\begin{proof}
In light of the characterization of the sheaves $\sF_n(X)$ given by Proposition \ref{sheaves F}, this is just how even-periodic Landweber exact homology theories are defined, see for instance \cite[Lecture 18, Proposition 6]{Lurie Chromatic}.
\end{proof}

With more hypothesis, we can exhibit $\E$-structures on Landweber spectra, coming from upgrading the underlying ordinary groups to sufficiently nice spectral formal groups. From another perspective, we relate Lurie's \textit{orientation classifier} construction $\mathfrak O_{\w{\G}}$ for a formal group $\w{\G}$  from \cite[Definition 4.3.14]{Elliptic 2}, with a Landweber exact spectrum under certain extra assumption.

Recall here that a $1$-dimensional formal group law $\w{\G}$ over an $\E$-ring $A$ is \textit{balanced}  in the sense of \cite[Definition 6.4.1]{Elliptic 2} if both of the following hold:
\begin{itemize}
\item The unit map $A\to\mathfrak O_{\w{\G}}$ induces an isomorphism of commutative rings $\pi_0(A)\simeq \mathfrak O_{\w{\G}}$.

\item The homotopy groups of $\mathfrak O_{\w{\G}}$ are concentrated in even degrees.
\end{itemize}

\begin{prop}\label{LEFTish}
Let $\w{\G}\in\mathrm{FGrp}_{\dim= 1}(A)$ be a balanced $1$-dimensional formal group over an $\E$-ring $A$, such that the underlying ordinary formal group $\w{\G}{}^0\in \mathrm{FGrp}_{\dim =1}(\pi_0(A))$ is Landweber exact. Then the orientation classifier $\mathfrak O_{\w{\G}}$ exhibits an $\E$-ring structure on the Landweber exact even-periodic spectrum $E_{\pi_0(A)}$ associated to $\w{\G}{}^0$.
\end{prop}

\begin{proof}
By definition, the $\E$-ring $\mathfrak O_{\w{G}}$ is complex periodic. Under the assumption that $\mathfrak O_{\w{\G}}$ is balanced, it follows from passing to underlying stacks from the commutative diagram of non-connective fpqc stacks
$$
\begin{tikzcd}
\Spec(\mathfrak O_{\w{\G}})\ar{d}{u'} \ar{r}{f'} &\M\ar{d}{u}\\
\Spec(A) \ar{r}{f}& \mathcal M_\mathrm{FG},
\end{tikzcd}
$$
that $(f')^\heart :\Spec(\pi_0(\mathfrak O_{\w{\G}}))\to(\M)^\heart\simeq \mathcal M^\heart_\mathrm{FG}$ recovers the map $f^\heart:\Spec(\pi_0(A))\to \mathcal M^\heart_\mathrm{FG}$ classifying $\w{\G}{}^0$. Landweber exactness of the latter means that $f^\heart$ is flat. On the other hand, since balancedness also implies that $\pi_{2i+1}(\mathfrak O_{\w{\G}})=0$ for all $i$, and quasi-coherent sheaves on $\M$ are all weakly $2$-periodic by design, it follows that the morphism $f'$ is flat itself. Just as in the proof of Proposition \ref{sheaves F}, we now find that
\begin{eqnarray*}
(E_{\pi_0(A)})_n(X) &\simeq& (f^\heart)^*(\pi_n(\sO_{\M}\o X))\\
&\simeq& \pi_n((f')^*(\sO_{\M}\otimes X))\\
&\simeq &  \pi_n(f^*(\sO_{\M}\o X)\\
&\simeq & \pi_n(\mathfrak O_{\w{\G}}\o X),
\end{eqnarray*}
proving that the spectrum $\mathfrak O_{\w{\G}}$ indeed represents the homology theory $E_{\pi_0(A)}$.
\end{proof}

\begin{ex}\label{What came before}
The proof of Proposition \ref{LEFTish} shows that, when complex periodic $\E$-rings arise from the orientation classifier construction, they may be obtain by pullback along the forgetful map $u:\M\to\mathcal M_\mathrm{FG}.$  This happens in a number of cases:
\begin{enumerate}
\item 
By \cite[Corollary 4.3.27]{Elliptic 2} and Snaith's Theorem \cite[Theorem  6.5.1]{Elliptic 2},  there is a canonical equivalence $\mathrm{KU}\simeq \mathfrak O_{\w{\G}_m}$ between the complex $K$-theory spectrum, and the orientation classifier of the multiplicative formal group $\w{\G}_m$ over the sphere spectrum $S$. The latter is a balanced formal group by \cite[Proposition 6.5.2]{Elliptic 2} (from which Lurie is able to deduce Snaith's Theorem), and so
 there is a canonical pullback square of non-connective spectral stacks
$$
\begin{tikzcd}
\Spec(\mathrm{KU})\ar{d} \ar{r} &\M\ar{d}{u}\\
\Spec(S) \ar{r}{\eta_{\w{\G}_m}}& \mathcal M_\mathrm{FG}.
\end{tikzcd}
$$

\item For any perfect field $\kappa$ of characteristic $p>0$, and any formal group $\w{\G}_0$ of finite height over $\kappa$, there exists by  \cite[Theorem 3.0.11]{Elliptic 2} the spectral deformation $\E$-ring $R^\mathrm{un}_{\w{\G}_0}$, supporting a universal deformation $\w{\G}$ of $\w{\G}_0$. This is a spectral enhancement of the better-known Lubin-Tate deformation ring, which is recovered on $\pi_0$. By \cite[Theorem 6.4.7]{Elliptic 2}, the formal group $\w{\G}$ is balanced. It hence follows that there is a pullback square
$$
\begin{tikzcd}
\Spf(E(\kappa, \w{\G}_0))\ar{d} \ar{r} &\M\ar{d}{u}\\
\Spf(R_{\w{\G}_0}^\mathrm{un}) \ar{r}{\eta_{\w{\G}}}& \mathcal M_\mathrm{FG},
\end{tikzcd}
$$
exhibiting the relationship between Lubin-Tate spectrum $E(\kappa, \w{\G}_0)\simeq \mathfrak O_{\w{\G}}$ and the non-connective spectral stack $\M$.

\item Though not quite fitting into the paradigm of Proposition \ref{LEFTish}, the construction of the $\E$-ring of topological modular forms from \cite[Chapter 7]{Elliptic 2} is analogous to the previous two examples. It start with the moduli stack $\mathcal M_\mathrm{Ell}^s$ of strict elliptic curves, which are in \cite[Definition 2.0.2]{Elliptic 1} defined as abelian group objects in varieties over $\E$-ring; see \cite[Definition 1.1.1]{Elliptic 1} for the latter. As discussed in \cite[Section 7.1]{Elliptic 2}, we may extract from any strict elliptic curve $E$ over an $\E$-ring $A$ a formal group $\widehat E$ over $A$. That gives rise to a map of stacks $\mathcal M^s_\mathrm{Ell}\to\mathcal M_\mathrm{FG}$,  classifying the formal group of the universal elliptic curve. By \cite[Theorem 7.3.1]{Elliptic 2}, this formal group is balanced. This implies, just like Proposition \ref{LEFTish}, that the stack of oriented elliptic curves, defined in \cite[Definition 7.2.9]{Elliptic 2}, fits into a pullback square
$$
\begin{tikzcd}
\mathcal M^\mathrm{or}_\mathrm{Ell}\ar{d} \ar{r} &\M\ar{d}{u}\\
\mathcal M^s_\mathrm{Ell} \ar{r}& \mathcal M_\mathrm{FG}
\end{tikzcd}
$$
 in the $\i$-category of non-connective spectral stacks. Though the main idea is the same, this is not quite an instance of Proposition \ref{LEFTish}. Indeed, the non-connective  (by \cite[Proposition 7.2.10]{Elliptic 2} Deligne-Mumford) spectral stack $\mathcal M^\mathrm{or}_\mathrm{Ell}$ has as its underlying ordinary stack being the usual stack of elliptic curves $\mathcal M_\mathrm{Ell}^\heart$, and is thus not affine. Of course, the main interest in it is that, thanks to \cite[Theorem 7.3.1]{Elliptic 2}, its ring of functions
$
\mO(\mathcal M^\mathrm{or}_\mathrm{Ell})\simeq \mathrm{TMF}
$
provides an approach to the $\E$-ring of topological modular forms.
\end{enumerate}

\end{ex}

\subsection{The Adams-Novikov spectral sequence}
We finally consider the descent spectral sequence on the non-connective spectral stack $\M$, and recover the ANSS, returning to the womb of chromatic homotopy theory.

\begin{prop}\label{ANSS}
The descent spectral sequence
$$
E^{s, t}_2 =\mathrm H^s\big(\mathcal M_\mathrm{FG}^\heart;\pi_t(\sO_{\mathcal M_\mathrm{FG}^\mathrm{or}})\big)\Rightarrow \pi_{t-s}(\mathcal O(\mathcal M_{\mathrm{FG}}^\mathrm{or}))
$$
is isomorphic to the ANSS 
$$
E^{s,t}_2 = \mathrm{Ext}^{s,t}_{\pi_*(\mathrm{MU}\otimes \mathrm{MU})}(\pi_*(\mathrm{MU}), \pi_*(\mathrm{MU}))\Rightarrow \pi_{t-s}(S).
$$
\end{prop}

\begin{proof}
It follows from the proof of Lemma \ref{Quillen} that fibered products of non-connective affines over $\mathcal M^\mathrm{or}_\mathrm{FG}$ is canonically equivalent to their categorical product, i.e.\ fibered products over the terminal object $\Spec(S)$. In particular,  the canonical map
\begin{equation}\label{cosimplicial spectra}
\check {\mathrm C}^\bullet (\Spec(\MP)/\mathcal M^\mathrm{or}_\mathrm{FG})\to \check {\mathrm C}^\bullet (\Spec(\MP)/\Spec(S))
\end{equation}
is an equivalence of simplicial objects.

By passing to global functions on the left-hand side of the equivalence \eqref{cosimplicial spectra}, we obtain the cosimplicial spectrum whose associated Bousfield-Kan spectral sequence is, according to the proof of Proposition \ref{DSS}, the descent spectral sequence.
Its second page is sheaf cohomology on the underlying ordinary stack of $\M$, which is $\mathcal M^\heart_\mathrm{FG}$ by Corollary \ref{Underlying ordinary stack is MFG}.

 By passing to global functions on the right-hand side of \eqref{cosimplicial spectra} though, we recover as
$$
\sO(\check {\mathrm C}^\bullet (\Spec(\MP)/\Spec(S)))\simeq \mathrm{MP}^{\otimes (\bullet+1)},
$$
the Amitsur complex (also known as the cobar complex) of $\mathrm{MP}$. That is precisely the cosimplicial spectrum whose Bousfield-Kan spectral sequence is the $\mathrm{MP}$-based Adams spectral sequence, see e.g.\ \cite[Section 2.2]{Green Book}, \cite[Section 5]{COCTALOS}, or \cite[Lecture 8]{Lurie Chromatic}.

It remains to show that the $\mathrm{MP}$-based and the $\mathrm{MU}$-based Adams spectral sequences agree. The $\mathbb E_2$-ring unit map $S\to S[\beta^{\pm 1}]$ induces a map of $\mathbb E_2$-rings $\mathrm{MU}\simeq \mathrm{MU}\o S\to \mathrm{MU}\o S[\beta^{\pm 1}]\simeq \mathrm{MP}$ induces a map on Amitsur complexes $\mathrm{MU}^{\o (\bullet +1)} \to\MP^{\o (\bullet +1)}$, which then gives rise to a map between the Adams spectral sequences. To prove that a map of spectral sequences is an equivalence, it suffices to show that this happens on the second page. That is a classical change-of-rings observation, see for instance \cite[Example 7.4]{Goerss on ANSS}, but it also follows from any other way of identifying the $E_2$-page of the Adams-Novikov (i.e.\ $\mathrm{MU}$-based Adams) spectral sequence with sheaf cohomology on the stack of formal groups $\mathcal M^\heart_\mathrm{FG}$, e.g.\ \cite[Lectures 10 \& 11]{Lurie Chromatic} or \cite[Remark 3.14]{Goerss}, since we already know the latter agree with the $E_2$-page of the $\mathrm{MP}$-based Adams spectral sequence thanks to the preceding discussion.
\end{proof}

\begin{remark}
In the proof of Proposition \ref{ANSS}, we used the fpqc cover $\Spec(\mathrm{MP})\to\M$ to obtain the descent spectral sequence. Thus it amounts to little more than a redressing of the usual connection between $\mathrm{MU}$ and formal group, stemming ultimately from Quillen's Theorem. But if we instead invoke Construction \ref{Cons of DSS}, we obtain the descent spectral sequence starting purely from the non-connective geometric stack $\M$. In particular, the complex bordism spectrum plays no distinguished role in setting up the ANSS from the perspective.
\end{remark}

\begin{remark}
A simple modification of the proof of Proposition \ref{ANSS} shows that for any spectrum $X$, the descent spectral sequence for the quasi-coherent sheaves $\sO_{\M}\otimes X$ on $\M$ gives rise to the ANSS for $X$. Using the conventions and results from Subsection \ref{LEFT}, its second page may be written as
$$
E^{s,t}_2 = \mathrm H^s(\mathcal M^\heart_\mathrm{FG}; \sF_t(X))
=
\mathrm H^s\big(\mathcal M^\heart_\mathrm{FG}; \omega_{\mathcal M^\heart_\mathrm{FG}}^{\otimes\left\lfloor \frac t2\right\rfloor}\otimes_{\sO_{\mathcal M^\heart_\mathrm{FG}}}\sF_{t(\mathrm{mod }2)}(X)\big).
$$
Specifying this back to the case of $X=S$, we find that the ANSS  may be rewritten in a way to highlight how the second page is determined purely by ordinary formal group data as
$$
E^{s,t}_2 = \mathrm H^s\big(\mathcal M^\heart_\mathrm{FG}; \omega^{\otimes t}_{\mathcal M^\heart_\mathrm{FG}}\big)\Rightarrow \pi_{2t-s}(S).
$$
\end{remark}

\section{Universal properties of periodic complex bordism}\label{Section 3}

So far, we have been using any form of periodic complex bordism $\MP$. In this subsection, we instead fix two specific $\E$-ring forms of $\MP$ that we discussed in Example \ref{The MPs}: the Thom spectrum $\mathrm{MUP}$ and the Snaith construction $\MP_\mathrm{Snaith} := (S[\mathrm{BU}])[\beta^{-1}]$. We give interpretations of the corresponding non-connective affine spectral schemes in terms of oriented formal groups. Finally we discuss a possible path towards using the ideas of this paper to prove Quillen's Theorem.

\subsection{Classical picture of coordinatized formal groups}

To motivative the discussion of universal properties of $\E$-forms of $\MP$, let us first recall how such an identification in terms of formal groups works on the level of $\pi_0$.

\begin{definition}[{\cite[Definition 5.3.5]{Elliptic 2}}]\label{coordinatized def}
Let $R$ be a commutative ring.
A  \textit{coordinatized formal group over $R$} consists of a pair $(\w{\G}, t)$ of a formal group $\w{\G}$ over $R$ and coordinate $t$ on it. Here a  \textit{coordinate on $\w{\G}$} means any  element $t\in\sO_{\w{\G}}(-e)$, whose image under the quotient map $\sO_{\w{\G}}(-e)\to\sO_{\w{\G}}/\sO_{\w{\G}}^2 =\omega_{\w{\G}}$ induces a trivialization of the dualizing line, i.e.\ an $R$-module isomorphism $R\simeq\omega_{\w{\G}}$. Coordinates on  $\w{\G}$ form a subset $\mathrm{Coord}_{\w{\G}}\subseteq \sO_{\w{\G}}(-e)$. 
\end{definition}

For any ring homomorphism $f : R\to R'$, the extension of scalars along $f$ preserves coordinates, and hence gives rise to a map $\mathrm{Coord}_{\w{\G}}\to \mathrm{Coord}_{f^*(\w{\G})}$. This defines a functor $\mathrm{Coord}(\w{\G}) :\CAlg^\heart_R\to\Set$, or equivalently, a map of ordinary stacks $\mathrm{Coord}(\w{\G})\to\Spec(R)$. From this perspective, Quillen's Theorem \cite[Theorem 5.3.10]{Elliptic 2} is the following assertion:

\begin{theorem}[Quillen]\label{Quillen's Theorem Statement}
 The commutative ring $L=\pi_0(\MP)$ corepresents the moduli of coordinatized formal groups. More precisely, for any formal group $\w{\G}$ over a commutative ring $R$, classified by a map of stacks
$\eta_{\w{\G}}:\Spec(R) \to \mathcal M_\mathrm{FG}^\heart$,
the fpqc cover $\Spec (L)\to\mathcal M_\mathrm{FG}^\heart$ participates in the pullback square of ordinary stacks
$$
\begin{tikzcd}
\mathrm{Coord}(\w{\G})\ar{d} \ar{r} &\Spec(L)\ar{d}\\
\Spec(R) \ar{r}{\eta_{\w{\G}}}& \mathcal M_\mathrm{FG}^\heart.
\end{tikzcd}
$$
\end{theorem}

\begin{remark}
The reader may be better accustomed to the claim that $\Spec(L)$ classifies formal group laws. As discussed in \cite[Remark 5.3.6]{Elliptic 2}, formal group laws are equivalent to coordinatized group laws up to equivalence. Indeed, the data of a coordinatization is equivalent to an isomorphism of pointed formal $R$-schemes $\widehat{\mathbf A}^1_R\simeq \w{\G}.$ We instead prefer the perspective of coordinatized formal groups given in Definition \ref{coordinatized def}, because it matches better with the characterizations of $\Spec(\MP)$ in Theorem \ref{What MUP is} and Proposition \ref{What Snaith is}.
\end{remark}

\begin{remark}
The notion a coordinatized formal group that we are using is the same as the one discussed in \cite[Section 2.3]{Goerss}. The prestack of coordinatized formal groups $\mathcal M_\mathrm{coord}$ from \cite[Definition 2.16]{Goerss}, however, disagrees with our meaning of the moduli of coordinatized formal groups. Indeed, while the objects of the groupoid $\mathcal M_\mathrm{coord}(R)$ are coordinatized formal groups, the morphisms are all isomorphisms of formal groups, not merely those which respect the chosen coordinatization.
\end{remark}

\subsection{Invertible and complex-exponentiable quasi-coherent sheaves}

To consider an analogue of the above story in spectral algebraic geometry, we must find an analogue of trivializing a line bundle. That necessitates a brief digression on invertible sheaves in spectral algebraic geometry.
In what follows, let $X$ be a non-connective spectral stack.

\begin{definition}[{\cite[Definition 2.9.5.1]{SAG}}]\label{invertible sheaves}
A quasi-coherent sheaf $\sL$ on $X$ is said to be \textit{invertible} if it is invertible in the symmetric monoidal $\i$-category $\QCoh(X)$. That is to say, if there exists a quasi-coherent sheaf $\mathscr L^{-1}$ such that $\mathscr L\otimes_{\sO_X}\mathscr L^{-1}\simeq \sO_X$. Invertible sheaves form a full subcategory $\mathcal P\mathrm{ic}^\dagger(X)\subseteq \QCoh(X)^\simeq$.
\end{definition}

\begin{remark}\label{line bundles in connective world}
When $X$ is connective, there exists another closely related notion: a \textit{line bundle on $X$} is such an invertible sheaf on $X$ for which both $\sL$ and $\sL^{-1}$ belong to the full subcategory $\QCoh(X)^\mathrm{cn}\subseteq\QCoh(X)$ of connective quasi-coherent sheaves. At least when $X$ is a spectral Deligne-Mumford stack, this is shown in \cite[Proposition 2.9.4.2]{SAG} to be equivalent to the quasi-coherent sheaf $\sL$ being locally free of rank $1$. In particular, if we denote by ${\mathcal P\mathrm{ic}}(X)\subseteq\mathcal P\mathrm{ic}^\dagger(X)$ the subspace of line bundles, then $\mathcal P\mathrm{ic}(X)$ recovers the usual Picard groupoid for an ordinary scheme $X$. The space $\mathcal P\mathrm{ic}^\dagger(X)$ is always larger however, as it contains invertible sheaves such as $\Sigma^n(\sO_X)$ for any $n\in \Z$.
\end{remark}

\begin{remark}
 Recall that the adjunction $\mS\rightleftarrows \Cat : (-)^\simeq$, between spaces viewed as $\infty$-groupoids and $\i$-categories, is Cartesian symmetric monoidal. It therefore induces an adjunction $\CMon(\mS)\rightleftarrows \CMon(\Cat) : (-)^\simeq$ between $\E$-spaces and symmetric monoidal $\i$-categories. In particular, the relative smash product symmetric monoidal structure on $\QCoh(X)$ equips the maximal underlying subspace $\QCoh(X)^\simeq$ with the structure of an $\E$-space. The subspace $\mathcal P\mathrm{ic}^\dagger(X)$ inherits a group-like $\E$-structure.
 \end{remark}

Let $\mathrm{Vect}^\simeq_{\mathbf C}$ denote the the topological category of finite-dimensional complex vector spaces and linear isomorphisms, viewed as an $\i$-category through an implicit application of the nerve construction of \cite[Definition 1.1.5.5]{HTT}. This is an $\i$-groupoid, explicitly $\Vect^\simeq_{\mathbf C}\simeq \coprod_{n\ge 0}\mathrm{BU}(n)$, and direct sum of vector spaces makes it into an $\E$-space.

\begin{definition}
Let $X$ be a non-connective spectral stack. A \textit{complex-exponentiable quasi-coherent sheaf on $X$} is a symmetric monoidal functor $ \mathrm{Vect}^\simeq_{\mathbf C}\to \QCoh(X)$. They form an $\i$-category $\QCoh^{\mathbf C}(X)$.
\end{definition}

\begin{remark}
Let  $\sF^{\otimes} : \mathrm{Vect}^\simeq_{\mathbf C}\to\QCoh(X)$ be a complex-exponentiable quasi-coherent sheaf. We will view the functor value $\sF=\sF^\o(\mathbf C)$ as the \textit{underlying quasi-coherent sheaf} of $\sF^{\otimes}$. By symmetric monoidality, this determines the functor object-wise as $\sF^{\otimes}(\mathbf C^n)\simeq \sF^{\otimes n}$. The complex-exponentiable structure amounts to specifying appropriately coherently compatible system of a $\mathrm U(n)$-action on $\sF^{\otimes n}$ for every $n\ge 0$. Said differently, a complex-exponentiable structure on an underlying quasi-coherent sheaf $\sF$ consists of a functorial system of powers $\sF^{\otimes V}\simeq \sF^{\otimes \dim_{\mathbf C}(V)}$ for finite dimensional complex vector spaces $V$.
\end{remark}

\begin{remark}
If the underlying quasi-coherent sheaf $\sF$ of a complex-exponentiable quasi-coherent sheaf is invertible, then there is a canonical symmetric monoidal factorization $\sF^{\otimes}:\Vect^\simeq_{\mathbf C}\to \mathcal P\mathrm{ic}^\dagger(X)\subseteq \QCoh(X)$. Because the $\E$-space $\mathcal P\mathrm{ic}^\dagger(X)$ is group-like, this further factors through the group completion $(\Vect_{\mathbf C}^\simeq)^\mathrm{gp}\simeq\Omega^\i(\mathrm{ku})$, see \cite[Section 6.5]{Elliptic 2}. That is to say, a complex-exponentiable invertible sheaf is equivalent to an $\E$-space map $\Omega^\i(\mathrm{ku})\to\mathcal P\mathrm{ic}^\dagger(X)$, or yet equivalently a map of connective spectra $\mathrm{ku}\to\mathfrak{pic}^\dagger(X)$.
\end{remark}

\subsection{Universal property of the Thom spectrum $\mathrm{MUP}$}
The notion of a complex-exponentiable quasi-coherent sheaf, introduced in the last section, will enable us to formulate the analogue of Theorem \ref{Quillen's Theorem Statement} for the Thom spectrum $\E$-ring $\mathrm{MUP}$.

\begin{prop}\label{Dualizing line has complex exponentials}
Let $\w{\G}$ be an oriented formal group over an $\E$-ring $A$. The dualizing line $\omega_{\w{\G}}$ admits a canonical enhancement to a complex-exponentiable quasi-coherent sheaf $\omega_{\w{\G}}^\otimes$ on $\Spec(A)$.
\end{prop}

\begin{proof}
 Let us write $X\simeq \Spec(A)$. The functor $\sO_X\otimes -:\Sp\to\QCoh(X)$ from Example \ref{Standard notation} is symmetric monoidal, and as such induces an (also symmetric monoidal) functor $\sO_X\otimes -:\mathcal P\mathrm{ic}^\dagger(S)\to \mathcal P\mathrm{ic}^\dagger(X)$ between invertible objects.
 Consider the $J$-homomorphism, viewed as map of $\E$-spaces $J:\mathrm{Vect}^\simeq_{\mathbf C}\to\mathcal P\mathrm{ic}^\dagger(S)$, given by $V\mapsto S^V =\Sigma^V(S)$. Combining these two, and the multiplicative inverse self-equivalence $(-)^{-1}: \mathcal P\mathrm{ic}^\dagger(S)\to\mathcal P\mathrm{ic}^\dagger(S)$, whose existence we owe to the fact that $\mathcal{P}\mathrm{ic}^\dagger(S)$ is a grouplike $\E$-space, we obtain a symmetric monoidal composite functor
 $$
\mathrm{Vect}^\simeq_{\mathbf C}\xrightarrow{J} \mathcal P\mathrm{ic}^\dagger(S)\xrightarrow{(-)^{-1}} \mathcal P\mathrm{ic}^\dagger(S)\xrightarrow{\sO_X\o -}  \mathcal P\mathrm{ic}^\dagger(X)\subseteq\QCoh(X).
$$
This exhibits an complex-exponentiable structure on the quasi-coherent sheaf $\Sigma^{-2}(\sO_X)$ on $X$. Because $\w{\G}$ is an oriented formal group, this is equivalent to the dualizing line $\omega_{\w{\G}}$.
\end{proof}

\begin{remark}
The proof of Proposition \ref{Dualizing line has complex exponentials} shows that the quasi-coherent sheaf $\Sigma^{2}(\sO_X)$ always admits a complex-exponentiable structure for any non-connective spectral stack $X$. Indeed, in the universal case $X=\Spec (S)$, the $J$-homomorphism may be viewed as exhibiting $\Sigma^2(S)$ to be complex-exponentiable.
\end{remark}

\begin{cons}\label{Constriv}
For any oriented formal group $\w{\G}$ over an $\E$-ring $A$, there are thanks to Proposition \ref{Dualizing line has complex exponentials} two complex-exponentiable quasi-coherent sheaves on $\Spec(A)$: the dualizing line $\omega^\o_{\G}$ and the trivial functor $\sO_{\Spec(A)}^\o$. Given any $\E$-ring map $f : A\to B$, the base-change $f^*\w{\G}$ is a formal group on $B$, and all the other structures in sight respect base-change. Let therefore
$
\mathrm{Triv}^\o_{\mathbf C}(\omega_{\w{\G}})\in\Fun(\CAlg_A, \mS)
$
be the presheaf given by
$$
B\mapsto \Map_{\QCoh^{\mathbf C}(\Spec(B))}^\simeq \big(\sO_{\Spec(B)}^\otimes, \,\omega_{f^*\w{\G}}^{\otimes}\big).
$$
That is to say, $\mathrm{Triv}^\o_{\mathbf C}(\omega_{\w{\G}})$ classifies trivializations of the dualizing line $\omega_{\w{\G}}$ as a complex-exponentiable quasi-coherent sheaf.
\end{cons}

\begin{theorem}\label{What MUP is}
Let $\w{\G}$ be an oriented formal group over an $\E$-ring $A$, classified by a map  of non-connective spectral stacks  $\eta_{\w{\G}}:\Spec(A)\to\M$. The canonical morphism $\Spec(\mathrm{MUP})\to\M$, stemming from complex-orientability of the Thom spectrum $\mathrm{MUP}$, induces a pullback square of non-connective stacks
$$
\begin{tikzcd}
\mathrm{Triv}^\o_{\mathbf C}(\omega_{\w{\G}})
\ar{d} \ar{r} &\Spec(\mathrm{MUP})\ar{d}\\
\Spec(A) \ar{r}{\eta_{\w{\G}}}& \M.
\end{tikzcd}
$$
\end{theorem}

\begin{proof}
There is an essentially unique map $\Spec(A)\to\M$ if and only if the $\E$-ring $A$ is complex-orientable, therefore the pullback statement will follow immediately from a description of the functor of points $\Spec(\mathrm{MUP}):\CAlg\to\mS$. That is to say, it suffices to verify that the canonical maps induce the homotopy pullback equivalence
$$ 
\Spec(\mathrm{MUP})(A)\times_{\M (A)}\{\w{\G}\} \simeq \Map_{\QCoh^{\mathbf C}(\Spec(A))}^\simeq (\sO_{\Spec(A)}^\otimes, \,\omega_{\w{\G}}^{\otimes}).
$$
in the $\i$-category of spaces.

By definition, the $\E$-ring $\mathrm{MUP}$ is defined as the Thom spectrum of the symmetric monoidal functor
$J:(\mathrm{Vect}^\simeq_{\mathbf C})^\mathrm{gp}\to\Sp.$ By the $\i$-categorical perspective on Thom spectra, as developed in \cite{ABGHR} and \cite{ABG}, an $\E$-ring map $\mathrm{MUP}\to A$ is therefore equivalent to an equivalence on  the $\i$-category of symmetric monoidal functors $\mathrm{Fun}^\o ((\Vect_{\mathbf C}^\simeq)^\mathrm{gp}, \Mod_A)$ between the composite
$$
(\Vect^\simeq_{\mathbf C})^\mathrm{gp}\xrightarrow{J}\Sp\xrightarrow{A\otimes -}\Mod_A
$$
and the corresponding constant functor with value $A$.

Note that, for instance on the account of the the idempotence of the $\E$-space map $(-)^{-1} : (\Vect^\simeq_{\mathbf C})^\mathrm{gp}\to (\Vect^\simeq_{\mathbf C})^\mathrm{gp}$, pre-composition with $(-)^{-1}$ induces a self-equivalence on  the $\i$-category $\mathrm{Fun}^\o ((\Vect_{\mathbf C}^\simeq)^\mathrm{gp}, \Mod_A)$. The constant functor with value $A$ is invariant under this self-equivalence, allowing us to conclude that an $\E$-ring map $\mathrm{MUP}\to A$ is also equivalent to the data of an equivalence in  $\mathrm{Fun}^\o ((\Vect_{\mathbf C}^\simeq)^\mathrm{gp}, \Mod_A)$ between
$$
(\Vect^\simeq_{\mathbf C})^\mathrm{gp}\xrightarrow{(-)^{-1}}(\Vect^\simeq_{\mathbf C})^\mathrm{gp}\xrightarrow{J}\Sp\xrightarrow{A\otimes -}\Mod_A
$$
and the constant functor with value $A$. Since $J$ is a map of group-like $\E$-spaces, there is a canonical commutative square of $\E$-spaces
$$
\begin{tikzcd}
(\Vect^\simeq_{\mathbf C})^\mathrm{gp} \ar{r}{(-)^{-1}} \ar{d}{J} & (\Vect^\simeq_{\mathbf C})^\mathrm{gp}\ar{d}{J}\\
\Sp \ar{r}{(-)^{-1}} & \Sp,
\end{tikzcd}
$$
which we may use to re-write the first of the two functors in question as
$$
(\Vect^\simeq_{\mathbf C})^\mathrm{gp}\xrightarrow{J}\Sp\xrightarrow{(-)^{-1}}\Sp\xrightarrow{A\otimes -}\Mod_A.
$$

Both of the symmetric monoidal functors in question take values in the full symmetric monoidal subcategory $\mathcal P\mathrm{ic}^\dagger(A)\subseteq \Mod_A$. It is therefore equivalent to look for an equivalence between the two functors inside the full subcategory of symmetric monoidal functors $(\Vect_{\mathbf C}^\simeq)^\mathrm{gp}\to\mathcal P\mathrm{ic}^\dagger(A)$. The universal property of group completion garners a homotopy equivalence
$$
\Map_{\CMon^\mathrm{gp}(\mS)}((\Vect^\simeq_{\mathbf C})^\mathrm{gp}, \mathcal P\mathrm{ic}^\dagger(A))\simeq 
\Map_{\CMon(\mS)}(\Vect^\simeq_{\mathbf C}, \mathcal P\mathrm{ic}^\dagger(A)),
$$
and the right-hand side once again embeds fully faithfully into $\mathrm{Fun}^\o (\Vect^\simeq_{\mathbf C}, \Mod_A)$. In conclusion, a map of $\E$-rings $\mathrm{MUP}\to A$ equivalently corresponds to the data of  an equivalence in the $\i$-category $\mathrm{Fun}^\o (\Vect^\simeq_{\mathbf C}, \Mod_A)\simeq \QCoh^{\mathbf C}(\Spec(A))$ between the composite functor
$$
\Vect^\simeq_{\mathbf C}\xrightarrow{J}\Sp\xrightarrow{(-)^{-1}}\Sp\xrightarrow{A\otimes -}\Mod_A
$$
and the constant functor with value $A$. We may recognize the first of these functors as exhibiting the complex-exponentiable structure on the dualizing line $\omega_{\w{\G}}$ by the proof of Proposition \ref{Dualizing line has complex exponentials}. The constant symmetric monoidal functor with value $A$ similarly encodes the complex-exponentiable structure on the trivial bundle $\sO_{\Spec(A)}^{\otimes}$, leading to the conclusion of the Theorem.
\end{proof}

\begin{remark}\label{What is MUP _really though_?}
The conclusion of Theorem \ref{What MUP is} may be expressed as the assertion that the non-connective spectral scheme $\Spec(\mathrm{MUP})$ parametrizes the data of an oriented formal group $\w{\G}$, together with a trivialization of the dualizing line $\omega_{\w{\G}}$ as a complex-exponentiable quasi-coherent sheaf. Informally, an $\E$-ring map $\mathrm{MUP}\to A$ amounts to specifying an oriented formal group $\w{\G}$ over $A$, and a system of equivalences $\omega_{\w{\G}}^{\otimes V}\simeq A$ in the $\i$-category $\Mod_A,$ equivariant and symmetric monoidal in $V\in\Vect^\simeq_{\mathbf C}$. Yet more explicitly, the trivialization data consists of an $\mathrm U(n)$-equivariant $A$-module equivalence $\theta_n : \omega_{\w{\G}}^{\otimes n}\simeq A$ for every $n\ge 0$, satisfying $\theta_n\otimes \theta_m\simeq \theta_{n+m}$ and $\theta_0\simeq \mathrm{id}_A$.
\end{remark}

\begin{remark}
In light of the proof of Theorem \ref{What MUP is}, the final informal description from Remark \ref{What is MUP _really though_?} may be upgraded, using the homotopy equivalence
\begin{equation}\label{space level splitting}
(\Vect^\simeq_{\mathbf C})^\mathrm{gp}\simeq \mathbf Z\times \mathrm{BU}.
\end{equation}
The trivialization data may therefore be extended to $\mathrm U$-equivariant $A$-module equivalences $\theta_n : \omega^{\o n}_{\w{\G}}\simeq A$ for all $n\in\Z$, satisfying $\theta_n\otimes \theta_m\simeq \theta_{n+m}$ and $\theta_0\simeq \mathrm{id}_A$ as before. But since the splitting \eqref{space level splitting} does not hold on the level of $\E$-spaces, making the compatibility and coherence precise is less straightforward in this formulation than that of  Remark \ref{What is MUP _really though_?}.
\end{remark}

\subsection{Connection with coordinatized formal groups}\label{Subsection 3.4. on coord}
We would like to relate the structure $\mathrm{triv}_{\mathbf C}^\o(\w{\G})$, highlighted in Theorem \ref{What MUP is}, with the notion of coordinatized formal groups from Definition \ref{coordinatized def}. First we must extend it to the spectral setting.

\begin{definition}\label{space of coordinates}
\textit{The space of coordinates} on a formal group $\w{\G}$ over an $\E$-ring $A$ is
$$\mathrm{Coord}_{\w{\G}}
\simeq
\Omega^\infty( \sO_{\w{\G}}(-e) )\times_{\Omega^\infty(\omega_{\w{\G}})} \Map_{\Mod_A}^\simeq(A,\, \omega_{\w{\G}}).
$$
\end{definition}

\begin{remark}
Equivalently, $\mathrm{Coord}_{\w{\G}}$ is the union of those path-connected components of the space $\Omega^\infty( \sO_{\w{\G}}(-e) ),$ which correspond to coordinates on the underlying formal group $\w{\G}{}^0$ over $\pi_0(A)$. That is because $\omega_{\w{\G}}$ is a flat $A$-module, hence an $A$-module map $A\to \omega_{\w{\G}}$ is an equivalence if and only if it induces an equivalence on $\pi_0$, where it gives rise to a $\pi_0(A)$-linear map $\pi_0(A)\to \omega_{\w{\G}{}^0}$. As consequence, the space $\mathrm{Coord}_{\w{\G}}$ is discrete whenever the base ring $A$ is, recovering its meaning from Definition \ref{coordinatized def}. 
\end{remark}

Consider the subcategory $\Vect_{\mathbf C}^{\dim = 1}\subseteq\Vect_{\mathbf C}^\simeq$ spanned by 1-dimensional complex vector spaces, and restriction of previously-discussed functors to this subcategory.

\begin{prop}\label{higher coordinatization}
Let $\w{\G}$ be an oriented formal group  over an $\E$-ring $A.$ There is a canonical homotopy equivalence
$$
\mathrm{Coord}_{\w{\G}} \simeq \Map^\simeq_{\Fun(\Vect_{\mathbf C}^{\dim =1}, \Mod_A)}(A^\otimes,\, \omega_{\w{\G}}^\o).
$$
\end{prop}

\begin{proof}
In light of the homotopy equivalence $\Vect_{\mathbf C}^{\dim = 1}\simeq\mathrm{BU}(1)$, an equivalence $A^\o \simeq\omega_{\w{\G}}^\o$ of functors $\Vect_{\mathbf C}^{\dim = 1}\to\Mod_A$ amounts to two things:
\begin{enumerate}[label = (\alph*)]
\item An $A$-module equivalence $\theta: A \simeq\omega_{\w{\G}}$.
\item A map of spectra $\tau:\Sigma^\infty(\mathrm{BU}(1))\to A$.
\end{enumerate}
These two are not unrelated, however. Restriction of $\tau$ along the inclusion
$$S^2\simeq \mathbf{CP}^1\subseteq\mathbf{CP}^\infty\simeq \mathrm{BU}(1)$$
gives rise to a map of $A$-modules $\Sigma^{2}(A)\to A$, which must, through the orientation equivalence $\Sigma^{-2}(A)\simeq\omega_{\w{\G}}$, induce the $A$-module isomorphism $\theta$. Conversely, since orientation of $\w{\G}$ implies that it is the Quillen formal group of the complex oriented $\E$-ring $A$, the identification $\omega_{\w{\G}}\simeq \Sigma^{-2}(A)$ extends to an equivalence $\sO_{\w{\G}}(-e)\simeq C^*_{\mathrm{red}}(\mathbf{CP}^\infty; A)$. Thus the map of spectra $\tau$ is equivalent to an $A$-linear map $t:A\to \sO_{\w{\G}}(-e)$.
In conclusion, an equivalence $A^\o \simeq\omega_{\w{\G}}^\o$ in the $\i$-category $\Fun(\Vect^{\dim =1}_{\mathbf C}, \Mod_A)$ is equivalent to:
\begin{enumerate}[label =($*$)]
\item An $A$-module map $t:A\to\sO_{\w{\G}}(-e)$, whose composite with the canonical map $\sO_{\w{\G}}(-e)\to \omega_{\w{\G}}$ induces a trivialization of the dualizing line of the formal group $\w{\G}$, which is to say, an $A$-module equivalence $A\simeq \omega_{\w{\G}}$.
\end{enumerate}
That is precisely the data of a coordinate on $\w{\G}$.
\end{proof}

As in the setting of ordinary algebraic geometry, the space of coordinates in the sense of Definition \ref{space of coordinates} is compatible with base-change along $\E$-ring maps $A\to B$, giving rise to a functor $\mathrm{Coord}(\w{\G}) : \CAlg_A\to\mS$. Under the canonical equivalence of $\i$-categories $\Fun(\CAlg_A, \mS)\simeq\Fun(\CAlg, \mS)_{/\Spec (A)},$ this may equivalently be seen as a map $\mathrm{Coord}(\w{\G})\to \Spec(A)$ of functors $\CAlg\to\mS$. This is analogous to the prestack $\mathrm{Triv}^\o_{\mathbf C}(\omega_{\w{\G}})$ from Construction \ref{Constriv}, and Proposition \ref{higher coordinatization} allows us to relate the two.

\begin{corollary}
Let $\w{\G}$ be an oriented formal group on an $\E$-ring $A$. There is a canonical map of non-connective spectral prestacks $\mathrm{Triv}_{\mathbf C}^\otimes(\omega_{\w{\G}})\to \mathrm{Coord}(\w{\G})$.
\end{corollary}

\begin{proof}
In light of Theorem \ref{What MUP is} and Proposition \ref{higher coordinatization}, this is induced by the composition 
$$
\QCoh_{\mathbf C}(\Spec(A))\simeq \Fun^\otimes(\Vect_{\mathbf C}^\simeq, \Mod_A)\to\Fun(\Vect_{\mathbf C}^\simeq, \Mod_A)\to\Fun(\Vect_{\mathbf C}^{\dim =1}, \Mod_A)
$$
of the forgetful functor from symmetric monoidal to non-symmetric-monoidal functors, with the functor restriction along the subcategory inclusion $\Vect_{\mathbf C}^{\dim = 1} \subseteq\Vect_{\mathbf C}^\simeq$.
\end{proof}

\begin{remark}\label{Discussion MUP v Coord}
A coordinate on an oriented formal group $\w{\G}$ over an $\E$-ring $A$ is by Proposition \ref{higher coordinatization} equivalent to an $\mathrm U(1)$-equivariant equivalence $\omega_{\w{\G}}\simeq A$. This gives rise to a symmetric monoidal system of $\mathrm U(1)^n$-equivariant equivalences $\omega_{\w{\G}}^{\otimes n}\simeq A$ for every $n\ge 0$. The additional structure encoded in a symmetric monoidal $\mathrm U(n)$-equivariant equivalences $\omega_{\w{\G}}^{\otimes n}\simeq A$ for every $n\ge 0$, equivalent according to the discussion of Remark \ref{What is MUP _really though_?} to an $\E$-map $\mathrm{MUP}\to A$ by Theorem \ref{What MUP is}, is therefore in extending the $\mathrm U(1)^n$-equivariance to an $\mathrm U(n)$-equivariance, along the diagonal matrix inclusion $\mathrm U(1)^n\subseteq \mathrm U(n)$.
\end{remark}

\subsection{Universal property of the Snaith construction for $\mathrm{MP}$} 

Another form of periodic complex bordism is the Snaith $\E$-ring $\MP_\mathrm{Snaith}:=(S[\mathrm{BU}])[\beta^{-1}]$. Here we use the series of inclusions
$$
S^2\simeq \mathbf{CP}^1\subseteq\mathbf{CP}^\infty\simeq \mathrm{BU}(1)\subseteq\mathrm{BU}
$$
inducing a map of spectra $\beta:\Sigma^2(S)\simeq \Sigma^\i(S^2)\to \Sigma^\infty(\mathrm{BU})\to S[\mathrm{BU}],$ which is equivalent to an element $\beta\in\pi_2(S[\mathrm{BU}])$. The Snaith $\E$-ring is not equivalent to the Thom spectrum $\mathrm{MUP}$ by \cite{Hahn-Yuan}. We can describe the map of non-connective spectral stacks $\Spec(\MP_\mathrm{Snaith})\to\M$ in a way somewhat analogous to the description of the map $\Spec(L)\to\mathcal M^\heart_\mathrm{FG}$ through Definition \ref{coordinatized def}. But compared to Theorem \ref{What MUP is}, the $\mathrm U$-action is incorporated in a somewhat \textit{ad hoc} manner.

\begin{cons}\label{Bott for Snaith}
Let $\w{\G}$ be an oriented formal group over an $\E$-ring $A$. Any fixed map of $\E$-rings $S[\mathrm{BU}]\to A$ induces an $A$-module map $\tau:A\to\omega_{\w{\G}}$ by
$$
A\xrightarrow{\Sigma^{-2}\beta}\Sigma^{-2}(A)\simeq \omega_{\w{\G}},
$$
where the unlabeled equivalence exhibits the orientation of $\w{\G}$. This gives rise to a canonical map of spaces $\Map_{\CAlg}(S[\mathrm{BU}], A)\to \Map_{\Mod_A}(A,\, \omega_{\w{\G}})$.
\end{cons}

\begin{prop}\label{What Snaith is}
Let $\w{\G}$ be an oriented formal group over an $\E$-ring $A$. The canonical map of non-connective spectral stacks $\Spec(\MP_\mathrm{Snaith})\to\M$, stemming from complex-orientability of the Thom spectrum $\mathrm{MUP}$, induces a homotopy pullback square
$$
\begin{tikzcd}
\Map_{\CAlg}(S[\mathrm{BU}],\, A)\times_{\Map_{\Mod_A}(A,\,\omega_{\w{\G}})}\Map^\simeq_{\Mod_A}(A, \, \omega_{\w{\G}})\ar{d} \ar{r} &\Spec(\MP_\mathrm{Snaith})(A)\ar{d}\\
\{\w{\G}\} \ar{r}& \M(A)
\end{tikzcd}
$$
in the $\i$-category of spaces.
\end{prop}

\begin{proof}
Once again, as there is an essentially unique map $\Spec(A)\to\M$ if and only if the $\E$-ring $A$ is complex-orientable, the pullback statement will follow immediately from a description of the functor of points $\Spec(\MP_\mathrm{Snaith}):\CAlg\to\mS$.

By the universal property of localization of $\E$-rings, e.g. \cite[Proposition 4.3.17]{Elliptic 2}, an $\E$-ring map $\MP_\mathrm{Snaith}=(S[\mathrm{BU}])[\beta^{-1}]\to A$ is equivalent to an $\E$-ring map $S[\mathrm{BU}]\to A$ for which the induced  $A$-module map $\beta:\Sigma^{-2}(A)\to A$ is an equivalence. In light of the definition of the map $\tau : A\to\omega_{\w{\G}}$ as
$$
\begin{tikzcd}
A\ar{rr}{\tau} \ar{dr}{\beta} & &  \omega_{\w{\G}}\\
& \Sigma^{-2}(A), \ar{ru}{\simeq}
\end{tikzcd}
$$
it is immediate that $\beta$ is invertible if and only if $\tau$ is. That proves the claim.
\end{proof}

\begin{remark}
Using the perspective of Subsection \ref{Subsection 3.4. on coord}, the universal property of the Snaith construction from Proposition \ref{What Snaith is} may be expressed in terms of coordinates. Indeed, a comparison between the  proof of Proposition \ref{higher coordinatization} and Construction \ref{Bott for Snaith} shows that  the map $\Map_{\CAlg}(S[\mathrm{BU}], A)\to \Map_{\Mod_A}(A, \omega_{\w{\G}})\simeq \Omega^\i(\omega_{\w{\G}})$ for $A$ a complex-oriented $\E$-ring with Quillen formal group $\w{\G}$, naturally factors through a map
$$
\Map_{\CAlg}(S[\mathrm{BU}],\, A)\to \Map_{\Sp}(\Sigma^\i(\mathbf{CP}^\i), A)\simeq \Omega^\i (\sO_{\w{\G}}(-e)).
$$
In light of this, the conclusion of Proposition \ref{What Snaith is} may now be expressed as
$$
\Map_{\CAlg}(\mathrm{MP}_\mathrm{Snaith}, A)\simeq \Map_{\CAlg}(S[\mathrm{BU}], A)\times_{\Omega^\i(\sO_{\w{\G}}(-e))} \mathrm{Coord}_{\w{\G}}.
$$
Similarly to Remark \ref{Discussion MUP v Coord} for the Thom cover $\Spec(\mathrm{MUP})\to \M$, the Snaith fpqc cover $\Spec(\MP_\mathrm{Snaith})\to \M$ thus extends the coordinate-discarding map $\mathrm{Coord}(\w{\G})\to \M$, for $\w{\G}$ the universal oriented formal group over $\M$, by incorporating an $\E$-action of the infinite unitary grop $\mathrm U$ that extends the $\mathrm U(1)$-action already present thanks to the orientation hypothesis.
\end{remark}

\begin{remark}
To encode the content of the previous Remark differently, consider  the functor from complex periodic $\E$-rings to spaces given by $A\mapsto\Omega^\i(\sO_{\w{\G}{}^\CMcal Q_A}(-e))$. Since the full subcategory of $\CAlg$ spanned by complex-periodic $\E$-rings is anti-equivalent to the $\i$-category of affine non-connective spectral stacks with a map to $\M$, this defines a non-connective spectral stack over $\M$. Indeed, if $\w{\G}$ is the universal complex oriented formal group, then the stack in question is the mapping stack $\underline{\mathcal M\mathrm{ap}}_{*}( \w{\G},\, \widehat{\mathbf A}^1 )$, which classifies pointed maps $\w{\G}\to\widehat{\mathbf A}^1_{\M}$ over $\M$. Then  Construction \ref{Bott for Snaith} may be seen as giving rise to a map of non-connective spectral stacks
$$\Spec( S[\mathrm{BU}])\times \M\to \Spec( \Sym^*(\Sigma^\infty(\mathbf{CP}^\i)))\times \M \simeq\underline{\mathcal M\mathrm{ap}}_{*}( \w{\G},\, \widehat{\mathbf A}^1 ),$$
and Proposition \ref{What Snaith is} expresses the universal property of the Snaith $\E$-ring as
$$
\Spec(\MP_\mathrm{Snaith})\simeq \big(\Spec( S[\mathrm{BU}])\times \M\big)\times_{\underline{\mathcal M\mathrm{ap}}_{*}( \w{\G},\, \widehat{\mathbf A}^1 )}\mathrm{Coord}(\w{\G}).
$$
\end{remark}

\subsection{Toward an algebro-geometric proof of Quillen's Theorem}
The key result we have used to prove the basic properties of the stack $\M$ is Quillen's Theorem, identifying $\pi_0(\MP)\simeq L$ with the Lazard ring, classifying coordinatized formal groups. It is through this theorem that complex bordisms really enter the discussion of either oriented formal groups in particular, or chromatic homotopy theory in general.

Recall that the standard proof of Quillen's Theorem, e.g.\ \cite[Part II]{Adams}, \cite[Section 4.1]{Green Book}, or \cite[Lectures 7 - 10]{Lurie Chromatic}, relies on combining two computations: Lazard's Theorem on $L$ and Milnor's Theorem on $\pi_0(\MP)=\pi_*(\MU)$, which identify both with polynomial rings in countably many variables. It is this ``matching of the two sides'' and the involved computations that go into proving these results that make Quillen's Theorem rather mysterious, even while it is an undeniable foundation of a large chunk of modern homotopy theory.

While our spectral algebro-geometric methods are sadly not capable of producing an alternative more insightful proof of Quillen's Theorem, we are able to isolate its geometric content. We show that it is equivalent to the following statement purely about the non-connective spectral stack $\M$, with no trace  of complex bordisms.

\begin{theorem}\label{Quillen Rephrased}
Quillen's Theorem (i.e.\ Theorem \ref{Quillen's Theorem Statement}) is equivalent to the assertion (of Corollary \ref{Underlying ordinary stack is MFG}) that the canonical map $\M\to\mathcal M_\mathrm{FG}$ induces an equivalence
$$
(\M)^\heart\simeq \mathcal M^\heart_\mathrm{FG}
$$
upon underlying ordinary stacks, exhibiting the ordinary stack of formal groups as the underlying ordinary stack of $\M$.
\end{theorem}

\begin{proof}
We have already seen above in the proof of Corollary \ref{Underlying ordinary stack is MFG} how the desired statement about underlying ordinary stacks follows from Quillen's Theorem.

To go the other way, let us assume that the map $\M\to\mathcal M_\mathrm{FG}$ induces an equivalence between underlying mapping stacks. Let the $\E$-ring $\MP$ be any form of periodic complex bordism. It is complex periodic, hence there is a canonical map of non-connective spectral stacks $\Spec(\MP)\to\M.$ Observe that we can write
\begin{equation}\label{colimit formula for M}
\M\simeq \varinjlim_{A\in \mC}\Spec(A),
\end{equation}
indexed over the full subcategory $\mC\subseteq\CAlg$ spanned by complex periodic $\E$-rings. Using this, we obtain a series of equivalences of non-connective spectral stacks
\begin{eqnarray*}
\Spec(\MP)
&\simeq&
\Spec(\MP)\times_{\M}\M\\
&\simeq&
\Spec(\MP)\times_{\M}\varinjlim_{A\in\mC}\Spec(A)\\
&\simeq&
\varinjlim_{A\in \mC}\Spec(\MP)\times_{\M}\Spec(A)\\
&\simeq &
\varinjlim_{A\in\mC}\Spec(\MP\o A),
\end{eqnarray*}
in which the third equivalence is due to the fact that $\Shv^\mathrm{nc}_\mathrm{fpqc}$ is an $\i$-topos and hence pullbacks in it are universal \cite[Theorem 6.1.0.6]{HTT}, while the fouth and final equivalence follows from the observation that fiber products of affines over $\M$ is given by the smash product, that we had originally made in the proof of Lemma \ref{affine diagonal lemma}. We may use this colimit formula for $\Spec(\MP)$, together with the fact that passage to the underlying ordinary stack by definition commutes with colimits, see Remark \ref{connective cover through colimits}, to conclude that
$$
\Spec(\pi_0(\MP))\simeq\varinjlim_{A\in\mC}\Spec(\pi_0(\MP\o A)).
$$
The colimit formula \eqref{colimit formula for M}, combined with the assumption on underlying ordinary stacks, similarly implies that
$$
\mathcal M_\mathrm{FG}^\heart\simeq (\M)^\heart\simeq\varinjlim_{A\in \mC}\Spec(\pi_0(A)).
$$
On the other hand, consider the ordinary moduli stack of coordinatized formal groups $\mathcal M_\mathrm{FG}^{\heart, \mathrm{coord}}$. It admits a canonical map $\mathcal M_\mathrm{FG}^{\heart, \mathrm{coord}}\to \mathcal M_\mathrm{FG}^{\heart}$ by discarding the choice of coordinate. The fiber of this map along any map of ordinary stacks $\Spec(R)\to \mathcal M_\mathrm{FG}^\heart,$ classifying a formal group $\w{\G}$ over the commutative ring $R$, is by definition
$$
\mathcal M_\mathrm{FG}^{\heart, \mathrm{coord}}\times_{\mathcal M_\mathrm{FG}^{\heart}} \Spec(R)\simeq \mathrm{Coord}(\w{\G}).
$$
Using the same reasoning as we did for $\Spec(\MP)$ above, we obtain
\begin{eqnarray*}
\mathcal M_\mathrm{FG}^{\heart, \mathrm{coord}}
&\simeq&
\mathcal M_\mathrm{FG}^{\heart, \mathrm{coord}}\times_{\mathcal M_\mathrm{FG}^\heart }\mathcal M_\mathrm{FG}^\heart \\
&\simeq&
\mathcal M_\mathrm{FG}^{\heart, \mathrm{coord}}\times_{\mathcal M_\mathrm{FG}^\heart }\varinjlim_{A\in\mC}\Spec(\pi_0(A))\\
&\simeq&
\varinjlim_{A\in \mC}\mathcal M_\mathrm{FG}^{\heart, \mathrm{coord}}\times_{\mathcal M_\mathrm{FG}^\heart }\Spec(\pi_0(A))\\
&\simeq &
\varinjlim_{A\in\mC}\mathrm{Coord}\big(\w{\G}{}^{\CMcal{Q}_0}_A\big).
\end{eqnarray*}
The one new ingredient, justifying the final equivalence, is the following observation. The map $\Spec (A)\to \M$ induces upon underlying ordinary stacks the map $\Spec(\pi_0(A))\to(\M)^\heart\simeq \mathcal M_\mathrm{FG}^\heart$, classifying the ordinary Quillen formal group $\w{\G}{}^{\CMcal{Q}_0}_A=\Spf(A^0(\mathbf{CP}^\i))$ of the complex periodic $\E$-ring $A$, see \cite[Notation 4.1.14]{Elliptic 2}.

The desired Theorem of Quillen, which is to say  that the  map $\Spec(\pi_0(\MP))\to\mathcal M_\mathrm{FG}^{\heart, \mathrm{coord}},$ induced by the canonical coordinate on $\w{\G}{}^{\CMcal{Q}_0}_{\MP}$, is an equivalence of ordinary stacks, now follows from the colimit formulas derived above, together with one classical homotopical computation. That is the fact \cite[Theorem 5.3.13]{Elliptic 2} that the canonical coordinate on $\w{\G}{}^{\CMcal{Q}_0}_{\MP}$ induces the equivalence of ordinary stacks
$$
\Spec(\pi_0(\MP\o A))\simeq \mathrm{Coord}\big(\w{\G}{}^{\CMcal{Q}_0}_A\big).
$$
for any complex periodic $\E$-ring (or even just commutative ring spectrum) $A$. For a proof, which does not require Quillen's Theorem, see  \cite[Proposition 7.11]{Rezk}, \cite[Lecture 7]{Lurie Chromatic}, or \cite[Proposition 6.5]{COCTALOS}.
\end{proof}

\begin{remark}
The idea behind Theorem \ref{Quillen Rephrased} is in some shape well-known to the experts. For instance, \cite{Peterson} is  (modulo the non-periodic setting) essentially a concise summary of the proof we have just given.
\end{remark}

Having a proof of Quillen's Theorem directly through Theorem \ref{Quillen Rephrased} would of course be wonderful, but it seems to be beyond our current reach. We sketch one possible approach below, reducing it to a concrete computation, collected in Claim \ref{claim}, yet one which we are completely incapable of tackling.

For this, we first recall that the canonical map $\M\to\mathcal M_\mathrm{FG}$ admits a factorization $\M\to\mathcal M^\mathrm{pre}_\mathrm{FG}\to\mathcal M_\mathrm{FG}$ through the stack of \textit{preoriented formal groups} $\mathcal M_\mathrm{FG}$. Here a
a \textit{preorientation} on a formal group $\w{\G}$ over an $\E$-ring $A$, in the sense of \cite[Definition 4.3.1]{Elliptic 2}, is a map of pointed spaces $S^2\to  \w{\G}(\tau_{\ge 0}(A))$. By a linearization procedure \cite[Construction 4.3.7]{Elliptic 2}, a preorientation gives rise to an $A$-module map $\omega_{\w{\G}}\to\Sigma^{-2}(A)$. Demanding that this map is an equivalence produces an orientation on $\w{\G}$, hence the map $\M\to\mathcal M^\mathrm{pre}_\mathrm{FG}$.  Quillen's Theorem in light of Theorem \ref{Quillen Rephrased} (i.e.\ the conclusion of Corollary \ref{Underlying ordinary stack is MFG}) now boils down to two facts:
\begin{enumerate}[label = (\roman*)]
\item The map $\mathcal M_\mathrm{FG}^\mathrm{pre}\to\mathcal M_\mathrm{FG}$ induces an equivalence upon underlying ordinary stacks.\label{First half of corollary}
\item The map  $\M\to\mathcal M^\mathrm{pre}_\mathrm{FG}$  induces an equivalence upon underlying ordinary stacks.\label{Second half of corollary}
\end{enumerate}
We offer two different proofs of the first one:

\begin{proof}[Proof of \ref{First half of corollary} (via abstract nonsense)]
Claim \ref{First half of corollary}, unlike \ref{Second half of corollary} concerns only connective spectral stacks, hence we may restrict to the connective setting for the length of discussing it.
By Remark \ref{Undst for conct}, the underlying ordinary stacks may thus be computed by restriction along the subcategory inclusion $\CAlg^\heart\subseteq\CAlg$. For any commutative ring $R$, the space $\mathcal M_\mathrm{FG}^\mathrm{pre}(R)$ consists of formal groups $\w{\G}\in \mathcal M_\mathrm{FG}(R)$, together with preorientation maps $S^2\to\w{\G}(R)$. But the space $\w{\G}(R)$ is discrete, hence $\pi_2(\w{\G}(R))=0$ and there is no non-trivial preorientation data.
\end{proof}

\begin{proof}[Proof of \ref{First half of corollary} (via explicit construction)]
Once again we restrict ourselves to the connective setting. Recall from \cite[Lemma 4.3.16]{Elliptic 2} that the stack of preoriented formal groups may be expressed in terms of the universal formal group $\w{\G}$ over $\mathcal M_\mathrm{FG}$ as the double based loop space of the zero-section $\mathcal M_\mathrm{FG}\to \w{\G}$ (viewed as a map of spectral stacks). That is to say, we have an equivalence of spectral stacks
$$
\mathcal M_\mathrm{FG}^\mathrm{pre}
\simeq
\mathcal M_\mathrm{FG}\times_{\mathcal M_\mathrm{FG}\times_{\w{\G}} \mathcal M_\mathrm{FG}}\mathcal M_\mathrm{FG}.
$$
On the other hand, since any formal group is formally affine, we have $\w{\G}\simeq \Spf_{\mathcal M_\mathrm{FG}}(\sO_{\w{\G}})$, hence the stack of preoriented formal groups may be written as
$$
\mathcal M_\mathrm{FG}^\mathrm{pre}\simeq \Spf_{\mathcal M_\mathrm{FG}}(\mathrm{HH}(\sO_{\mathcal M_\mathrm{FG}}/\sO_{\w{\G}}))
$$
 in terms of Hochschild homology $\mathrm{HH}(A/R)\simeq A\otimes_{A\otimes_R A}A$. Its underlying ordinary stack is therefore
$$
(\mathcal M_\mathrm{FG}^\mathrm{pre})^\heart\simeq \Spf_{\mathcal M^\heart_\mathrm{FG}}(\pi_0(\mathrm{HH}(\sO_{\mathcal M_\mathrm{FG}}/\sO_{\w{\G}}))).
$$
For any pair of connective $\E$-ring map $R\to A$ we have
$$\pi_0(\mathrm{HH}(A/R))\simeq \mathrm{HH}_0(\pi_0(A)/\pi_0(R))\simeq \pi_0(A),$$
hence the connectivity assumption implies that $\pi_0(\mathrm{HH}(\sO_{\mathcal M_\mathrm{FG}}/\sO_{\w{\G}}))\simeq\sO_{\mathcal M_\mathrm{FG}^\heart},$ proving assertion \ref{First half of corollary}.
\end{proof}

In order to discuss assertion \ref{Second half of corollary}, we must recall how to construct $\M$ from $\mathcal M_\mathrm{FG}^\mathrm{pre}$. Let $\w{\G}^\mathrm{pre}$ denote the universal preoriented formal group over $\mathcal M_\mathrm{FG}^\mathrm{pre}$. Its preorientation gives rise to the universal Bott map $\beta:\omega_{\w{\G}^\mathrm{pre}}\to\Sigma^{-2}(\sO_{\mathcal M^\mathrm{pre}_\mathrm{FG}})$, which must be inverted in order to make $\w{\G}^\mathrm{pre}$ oriented and hence pass to $\M$. That is to say, by \cite[Proposition 4.3.17]{Elliptic 2}, the non-connective stack of oriented formal groups is obtained from the stack of preoriented ones by
$$
\M\simeq \Spec_{\mathcal M_\mathrm{FG}^\mathrm{pre}}\Big(\varinjlim\big( \sO_{\mathcal M_\mathrm{FG}^\mathrm{pre}}\xrightarrow{\beta} \Sigma^{-2}(\omega_{\w{\G}^\mathrm{pre}}^{-1})
\xrightarrow{\beta}  \Sigma^{-4}(\omega_{\w{\G}^\mathrm{pre}}^{-1})^{\o 2} \xrightarrow{\beta} \Sigma^{-6}(\omega_{\w{\G}^\mathrm{pre}}^{-1})^{\o 3}\to\cdots \big)\Big).
$$
To find the underlying ordinary stack, we therefore need to determine $\pi_0$ of the quasi-coherent sheaf $\sO_{\M}$ on $\mathcal M^\mathrm{pre}_\mathrm{FG}$. The use of continuity of the functor $\pi_0$, flatness of the dualizing sheaf, and the explicit Hochschild formula for $\sO_{\mathcal M^\mathrm{pre}_{\mathrm{FG}}}$ from the second proof above, combine to show that  \ref{Second half of corollary} (and therefore Quillen's Theorem) is equivalent to:

\begin{claim}\label{claim}
Let $\w{\G}$ denote the universal formal group on the spectral stack $\mathcal M_\mathrm{FG}$. The canonical map of ordinary quasi-coherent sheaves on $\mathcal M^\heart_\mathrm{FG}$
$$
\sO_{\mathcal M_\mathrm{FG}^\heart}\to \varinjlim_{n\to\infty} \mathrm{HH}_{2n}(\sO_{\mathcal M_\mathrm{FG}}/\sO_{\w{\G}})\o_{\sO_{\mathcal M^\heart_\mathrm{FG}}} (\omega^{-1}_{\mathcal M_\mathrm{FG}^\heart})^{\o n}
$$
is an isomorphism.
\end{claim}

Attempting to prove this claim directly would however require a greater degree of understanding of the spectral moduli stack of formal groups $\mathcal M_\mathrm{FG}$, and the universal formal group $\w{\G}$ over it, than we currently possess.

\end{document}